\documentclass[sn-mathphys,Numbered]{sn-jnl}
\usepackage{graphicx}%
\usepackage{amsmath,amssymb,amsfonts}%
\usepackage{amsthm}%
\usepackage{xcolor}%
\usepackage{manyfoot}%
\usepackage{mathtools}
\usepackage{siunitx}

\newcommand{\CCC}{\mathbb{C}}

\newcommand{\RR}{\mathbb{R}}

\newcommand{\LL}{\mathcal{L}}
\newcommand{\e}{\varepsilon}
\newcommand{\ee}{\mathrm{e}}

\newcommand{\p}{\rho}
\newcommand{\tp}{\tilde{\rho}}
\newcommand{\OO}{\Omega}

\newtheorem{theorem}{Theorem}

\newtheorem{lemma}{Lemma}
\newtheorem{remark}{Remark}%

\begin{document}

\title{Systematic design of compliant morphing structures: a phase-field approach}

\author[1]{\fnm{Jamal} \sur{Shabani}}\email{shabanij@mcmaster.ca}
\author[2]{\fnm{Kaushik} \sur{Bhattacharya}}\email{bhatta@caltech.edu}
\author*[1]{\fnm{Blaise} \sur{Bourdin}}\email{bourdin@mcmaster.ca}

\affil[1]{\orgdiv{Department of Mathematics \& Statistics}, \orgname{McMaster University}, \orgaddress{\city{Hamilton, ON}, \country{Canada}}}
\affil[2]{\orgname{California Institute of Technology}, \orgaddress{\city{Pasadena, CA}, \country{USA}}}

\abstract{We investigate the systematic design of compliant morphing  structures composed of materials reacting to an external stimulus.
We add a perimeter penalty term to ensure existence of solutions.
We propose a phase-field approximation of this sharp interface problem, prove its convergence as the regularization length approaches 0 and present an efficient numerical implementation.
We illustrate the strengths of our approach through a series of numerical examples.
}



\maketitle

\section{Introduction}\label{introduction}
Advances in additive manufacturing and synthesis of complex ``active'' materials whose properties can be altered through external stimuli are opening the door to a new generation of integrated devices and materials.
While manufacturing such structures or materials has received a considerable attention (see for instance~\cite{Schaedler-Jacobsen-EtAl-2011b,Xu-Zhang-EtAl-2019a}), their actual design remains challenging.
Starting from the pioneering work of~\cite{Sigmund-1997,Larsen-Sigmund-EtAl-1997,Jonsmann-Sigmund-EtAl-1999}, topology optimization has established itself as a powerful tool for systematic design of micro-devices, MEMS, or materials microstructures.
Here, the goal is to algorithmically find the distribution of materials in a ground domain that optimizes an objective function~\cite{BendsoeSigmund2003}.
It is well-known that such problems generally do not admit a ``classical'' solution (see~\cite{Allaire2002} for instance) resulting in optimal designs consisting of an infinitely fine mixture of multiple materials.
Homogenization approaches~\cite{Cherkaev2000,ABFJ1997,Allaire2002} tackle this problem directly by extending admissible designs to such mixtures.
This type of approach is mathematically well grounded and leads to well-posed problems that can be implemented efficiently.
However, it is often criticized for leading to designs that
cannot be manufactured.
Several other classes of techniques aim at restricting the class of admissible designs in such a way that avoids fine mixtures.
The combination of material interpolation (SIMP) and filters~\cite{Bendsoe-Sigmund-1999,Bourdin-2001} is a commonly employed approach. 
Shape parameterization by level set functions~\cite{Allaire_2004,Allaire_2013} also limits the complexity of designs.
Finally, by penalizing the length (or surface) of interfaces between materials, perimeter penalization~\cite{AmbrosioButtazzo1993,Haber-EtAl-1996,Petersson-1999b} also produces designs with limited complexity.
Additionally, perimeter penalization can be efficiently implemented using a phase-field approach~\cite{BourdinChambolle2003,BourdinChambolle2006,Tran-Bourdin-2022c}.

In this article, we propose a phase-field algorithm for the systematic design of active structures achieving prescribed deformations under some unknown distributions of a stimulus.
Our focus is on linear elastic materials in which an external stimulus can generate an isotropic inelastic strain, similar to linear thermo-elastic materials.

Section~\ref{problem-statement} is devoted to the mathematical analysis of the problem and its phase-field approximation.
A numerical scheme is proposed in Section~\ref{sec:numerical-implementation} and illustrated by a series of numerical simulations in Section~\ref{sec:numerical-results}.

\section{Problem statement}\label{problem-statement}
Consider linear elastic materials whose constitutive laws depend on an external real-valued stimulus $s \in [-1,1]$  inducing an inelastic strain, \emph{i.e.}
\begin{equation}
    \label{eq:responsiveConstitutiveLaw}
    \sigma = \mathbb{C} \left(\ee(u) - \beta s\mathrm{I}_d\right)
\end{equation}
where $\mathbb{C}$ denotes the Hooke's law, $\ee(u) = (\nabla u + \nabla u^T)/2$ is the linearized strain associated to a displacement field $u$, $\beta \ge 0$ is a given parameter and $\mathrm{I}_d$ the $d\times d$ identity matrix.
Throughout this article, we call such a material a \emph{responsive material} characterized by $\mathbb{C}$ and $\beta$.

Consider a bounded open domain $\OO \subset \mathbb{R}^d$, $d = 2,3$, and an open subset $\OO_0$ of $\OO$. 
A \emph{design} $\mathrm{D}$ is a partition of $\OO$ into $m$ subdomains $(D_1,\dots, D_m)$ occupied by $m$ responsive materials characterized by $(\mathbb{C}_1,\beta_1),\dots,(\mathbb{C}_m,\beta_m)$.
Let $\Gamma_D \subset \partial \OO$ be a regular-enough part of the boundary of our domain with non-zero length and $\Gamma_N = \partial \OO \setminus \Gamma_D$.
Consider $n$ prescribed displacement fields $(\bar{u}_1,\dots , \bar{u}_n) \in \left[H^1(\OO_0; \mathbb{R}^d)\right]^n$ defined over $\OO_0$.

Our goal is to design a structure $\mathrm{D}$ and a family of stimulus functions $\mathrm{s}:=(s_1, \dots , s_n)$ such that in the equilibrium configuration associated with each stimulus $s_j$, $\OO_0$ is mapped to a region as close as possible to $\bar{u}_j(\OO_0)$, $j = 1,\dots,n$.
More precisely, let $(\theta_1,\dots,\theta_m)$ such that $\sum_{i=1}^m \theta_i = 1$ be a set of prescribed volume fractions in $[0,1].$
The space of admissible designs consists of partition of $\OO$ in subsets with prescribed volume fraction, \emph{i.e.}
\begin{multline}
    \label{eq:defD}
    \mathcal{D} := \Bigl\{(D_1,\dots,D_m); \bigcup_{1\le j \le m} \bar{D}_i = \OO,\ 
    D_i \cap D_j = \emptyset\ 1\le i < j \le m, \\
    |D_i| = \theta_i|\OO|,\ i = 1,\dots,m\Bigr\}.
\end{multline}
We consider the space of admissible stimuli 
$\mathrm{s}$ taking values in $[-1,1]^n$, \emph{i.e.}
\begin{equation}
    \label{eq:defS}
    \mathcal{S}:= L^1\left(\OO,[-1,1]^n\right).
\end{equation}
Given a design $\mathrm{D}$ and a set of stimuli $\mathrm{s}$, we define
\begin{equation}
    \label{eq:objectiveFunctionNoPerimeter}
    \mathcal{I}(u_1,\dots,u_n) := \sum_{j=1}^n\frac{1}{2}\int_{\OO_0} |u_j(x) - \bar{u}_j(x)|^2\, dx,
\end{equation}
where the $u_j \in V$, $1 \le j \le n$,  satisfy the weak form of the linearized elasticity system 
\begin{equation}
    \label{eq:stateSharpWeak}
    \int_{\OO} \sum_{i=1}^m \chi_{D_i}\mathbb{C}_i\left(\ee(u_j) - \beta_i s_j\mathrm{I}_d\right)\cdot \ee(\phi)\, dx = 0,\forall \phi \in V,
\end{equation}
where 
\begin{equation}\label{eq:sobolevSpace}
    V:= \left\{ \phi \in H^1(\Omega);\, \phi = 0 \text{ on } \Gamma_D \right\}.
\end{equation}

%

Let 
\begin{equation}
    \label{eq:defP}
    \mathcal{P}(\mathrm{D}) := \frac12 \sum_{i, j=1}^m \mathcal{H}^{d-1}\left(\partial^* D_i \cap \partial^* D_j \cap \Omega\right),
\end{equation}
where $\mathcal{H}^{d-1}$ denotes the $d-1$--dimensional Hausdorff measure and $\partial^*D$ the essential boundary of a set $D$ (a generalization of the classical concept of the boundary of a set that does not require regularity, see~\cite{Ambrosio-EtAl-2000} Definition 3.60).  
Given a small regularization parameter $\alpha > 0$, we study the problem
\begin{equation}
    \label{eq:topOptStrong}
    \inf_{(\mathrm{D},\mathrm{s}) \in \mathcal{D} \times \mathcal{S}} \mathcal{I}(u_1,\dots,u_n) + \alpha\mathcal{P}(\mathrm{D}).
\end{equation}

\subsection{Phase-field regularization}\label{phase-field-approach}
The phase-field approach to optimal design, introduced in~\cite{BourdinChambolle2003,BourdinChambolle2006} (see also~\cite{Wang-Zhou-2004a,Zhou-Wang-2007,Tran-Bourdin-2022c,Garcke-Huttl-EtAl-2021a,Garcke-Huttl-EtAl-2024a}) is based on the idea of variational approximation of the perimeter penalty $\mathcal{P}$ for generalized designs.
We define a set of \emph{generalized designs}, that is vector-valued functions $\rho\in \mathcal{D}_\rho$, where   
\begin{multline}
    \label{eq:defDrho}
    \mathcal{D}_\rho:=\Bigl\{(\p_1,\p_2,\ldots,\p_m) \in \left[H^1(\OO;[0,1])\right]^m,\\
    \sum_{i=1}^m\p_i=1,\int_\OO \rho_i\, dx = \theta_i|\OO|,\ 1 \le i \le m\Bigr\}.
\end{multline}
Loosely speaking, the components $\rho_i$ of the vector-valued phase-field $\rho$ can be thought of as a density of material $i$ at each point of the domain $\OO$, and classical designs would correspond to the situations where $\rho_i = \chi_{D_i}$.
Indeed, if $\mathrm{D} \in \mathcal{D}$ is a classical design, then $\left( \chi_{D_1},\dots,\chi_{D_m}\right)$ is a generalized design.
We then extend $\mathcal{P}$ to generalized designs by defining
\begin{equation}
    \label{eq:defPe}
    \mathcal{P}_\e(\p) := \int_\Omega \frac{W(\p)}{\e} + \e |D\p|^2\, dx,
\end{equation}
where $\e>0$ is a regularization parameter and $W$ is a non-negative function vanishing only at the vertices $p_1,\dots, p_m$ of the $m$-dimensional unit simplex and satisfying
\begin{multline}
    \label{eq:normalizationW}
    d_{ij}:= \inf \Biggl\{\int_0^1 W^{1/2}(\gamma(t))|\gamma'(t)|\,dt; \gamma \in C^{1}((0,1);\RR^m),\gamma(0)=\p_i,\gamma(1)=\p_j \Biggr\} = 1
\end{multline}
for all $1\le i < j \le m$.
Relation~\eqref{eq:normalizationW} is designed to recover the proper weights in the perimeter functional in the sharp-interface limit $\e \to 0$. It is typically recovered from the construction of the ``optimal profile'' problem (see~\cite{Alberti-2000}, for instance).

Next, we introduce a convex continuous function $a$ such that $a(0) = 0$ and $a(1) = 1$, so that the equilibrium equation~\eqref{eq:stateSharpWeak} can then be extended to generalized designs by 
\begin{equation}
    \label{eq:statePFWeak}
    \sum_{i=1}^m\int_\OO a(\p_i)\mathbb{C}_i\left(\ee(u_j)-\beta_i s_j\mathrm{I}_d\right)\cdot \ee(\phi)\, dx = 0\ \forall \phi \in V,\  1 \le j \le n.
\end{equation} 
For a given $\e>0$, the phase field  regularization of~\eqref{eq:topOptStrong} is then
\begin{equation}
    \label{eq:topOptPF}
    \inf_{(\rho,\mathrm{s}) \in \mathcal{D_\rho} \times \mathcal{S}} \mathcal{I}(u_1,\dots,u_n) + \alpha\mathcal{P}_\e(\rho).
\end{equation}

We show in the next Section that both problems admit solutions and that the solutions of~\eqref{eq:topOptStrong} converge in some sense to that of~\eqref{eq:topOptPF}.
This result justifies numerical approach presented below and which consists in minimizing~\eqref{eq:topOptPF} for ``small'' values of the regularization parameter $\e$.

\subsection{Existence of solutions}\label{existence-of-solutions}
The existence of minimizers for the regularized problem~\eqref{eq:topOptPF} and their convergence to that of the perimeter-controlled topology optimization problem~\eqref{eq:topOptStrong} is a relatively straightforward consequence of the now-classical $\Gamma$--convergence result for phase transition problems~\cite{Modica-Mortola-1977b,Modica1987a,Baldo-1990a,Alberti-2000}.

Prior to stating our main result, we need to introduce a few notations.
Let
\begin{equation}
    \label{eq:defPeT}
    \widetilde{\mathcal{P}}_\e(\p,\mathrm{s}) := \begin{cases}
        \mathcal{P}_\e(\p) & \text{ if } (\p,\mathrm{s}) \in \mathcal{D}_\p\times\mathcal{S}\\
        +\infty & \text{ otherwise,}
    \end{cases}
\end{equation}
and 
\begin{equation}
    \label{eq:defPT}
    \widetilde{\mathcal{P}}(\p,\mathrm{s}) := \begin{cases}
        \mathcal{P}(\p) & \text{ if } (\p,\mathrm{s}) \in \widetilde{\mathcal{D}}\times\mathcal{S}\\
        +\infty & \text{ otherwise,}
    \end{cases}
\end{equation}
where
\begin{equation}
    \label{eq:defDT}
    \widetilde{\mathcal{D}} = \left\{\p;  \exists (D_1,\dots,D_m)\in \mathcal{D}, \rho_i = \chi_{D_i},\ 1\le i \le m\right\},
\end{equation}
and \begin{equation}
    \label{eq:defIT}
    \widetilde{\mathcal{I}}(\rho,\mathrm{s}) = \mathcal{I}\left(u_1(\rho,s_1), \dots, u_n(\rho,s_n)\right),
\end{equation}
where $u_j(\rho)$, $1 \le j \le n$ satisfy~\eqref{eq:statePFWeak}.
We are now able to state the main existence and approximation result:
\begin{theorem}
    \label{thm:main}
    Let $\mathbb{C}_1,\dots,\mathbb{C}_m$ be symmetric definite linear operators over $\mathrm{M}^{d\times d}_\mathrm{sym}$ and assume that $\beta_i \ge 0$ for any $1 \le j \le m$.

    For any given $\e>0$, the problem
    \begin{equation}
        \label{eq:topOptPFT}
        \inf_{(\rho,\mathrm{s}) \in \mathcal{D_\rho} \times \mathcal{S}} \widetilde{\mathcal{I}}(\p,\mathrm{s}) + \alpha\widetilde{\mathcal{P}}_\e(\rho,\mathrm{s}),
    \end{equation}
    admits a solution $(\p_\e,\mathrm{s}_\e)$.
    Furthermore, there exists $(\p,\mathrm{s}) \in \widetilde{\mathcal{D}} \times \mathcal{S}$ such that a subsequence $\p_{\e'} \to \p$ in $\left[L^1(\OO)\right]^m$ and $\mathrm{s}_{\e'} \rightharpoonup \mathrm{s}$ in $\left[L^p(\Omega)\right]^n$ for any $1<p\le \infty$ and $(\p,\mathrm{s})$ is a solution of
    \begin{equation}
        \label{eq:topOptStrongT}
        \inf_{(\p,\mathrm{s}) \in \mathcal{D}_\rho \times \mathcal{S}} \widetilde{\mathcal{I}}(\p,\mathrm{s}) + \alpha\widetilde{\mathcal{P}}(\p,\mathrm{s}).
    \end{equation}
\end{theorem}

Before proving Theorem~\ref{thm:main}, we state and prove several preliminary lemmas.

\begin{lemma}[Equi-coercivity of the displacements]
    \label{lem:equicoercivity}
    Let $(\rho,\mathrm{s}) \in \mathcal{D}\times \mathcal{S}$ and $k_1,k_2,k_3 >0$ be such that for any $1 \le i \le m$ and for any $\Psi \in \mathrm{M}^{d\times d}_\mathrm{sym}$
    \begin{equation}
        \label{eq:boundsAi}
        k_1 \Psi\cdot \Psi \le \CCC_i \Psi\cdot \Psi \le k_2 \Psi\cdot \Psi,
    \end{equation}
    and for any $1\le i  \le m$ ,
    \begin{equation}
        \label{eq:boundsBetai}
        |\beta_i| \le k_3.         
    \end{equation}
    
    There exists $C>0$ such that if $(u_1,\dots u_n) \in V^n$ satisfies~\eqref{eq:statePFWeak}, then
    \[
        \|u_j\|_{H^1(\Omega)} \le C ,\  1 \le j \le n.
    \] 
\end{lemma}
\begin{proof}
    Using $u_j$ as the test function in~\eqref{eq:statePFWeak}, we obtain 
    \begin{equation*}
        \int_\Omega \sum_{i=1}^m a(\p_i)\mathbb{C}_i\ee(u)\cdot\ee(u)\, dx = 
        \int_\Omega \sum_{i=1}^m a(\p_i)\beta_i s \CCC_i \ee(u) \cdot \mathrm{I}_d\, dx.
    \end{equation*}
    Using then~\eqref{eq:boundsAi},  and Cauchy-Schwarz inequality, we get 
    \[
        k_1\sum_{i=1}^m a(\rho_i) \|\ee(u)\|^2_{L^2(\OO)} \le \|\ee(u)\|_{L^2(\OO)}\left\|\sum_{i=1}^m a(\rho_i)\beta_i \CCC_i \mathrm{I}_d\right\|_{L^2(\OO)}.
    \]
    Since $a$ is convex on $[0,1]$, $0 \le \rho_i \le 1$, and $\sum_i \rho_i = 1$ we have
    \[
        \sum_{i=1}^m a(\rho_i) \ge a\left(\sum_{i=1}^m \rho_i\right) = a(1) = 1.
    \] 
    Furthermore, since $a(\rho_i) <1$ for all $i$, we conclude that 
    \[
        k_1 \|\ee(u)\|^2_{L^2(\OO)} \le k_3\|\ee(u)\|_{L^2(\OO)}\left\|\sum_{i=1}^m \mathbb{C}_i\mathrm{I}_d\right\|_{L^2(\OO)},
    \]
    hence that 
    \[
        k_1 \|\ee(u)\|_{L^2(\OO)} \le C,
    \]
    for some $C>0$. We then conclude using Korn's inequality with boundary conditions~\cite[Theorem 6.15-4]{Ciarlet-2013a}.
\end{proof}


\begin{lemma}[Continuity of displacements]
    \label{lem:continuity}
    Let $\CCC_i$, $\beta_i$ be as in Lemma~\ref{lem:equicoercivity}.
    Consider a sequence $(\p_\e,\mathrm{s}_\e)_\e \in \mathcal{D}_\p \times S$ of designs and stimuli and $(\p,\mathrm{s}) \in \widetilde{\mathcal{D}}\times\mathcal{S}$ be such that $\p_\e \to \p,$ in $\left[L^1(\Omega)\right]^m$ and $\mathrm{s}_\e \rightharpoonup \mathrm{s}$ in  $\left[L^2(\Omega)\right]^n$. Let $u_\e = (u_{1,\e},\dots,u_{n,\e})$ (resp. $u = (u_1,\dots,u_n)$) be the equilibrium displacements associated with $(\p_\e,\mathrm{s}_\e)$ (resp. $(\p,\mathrm{s})$), given by~\eqref{eq:statePFWeak} (resp.~\eqref{eq:stateSharpWeak}).
    Then if the $\CCC_i$ and $\beta_i$ satisfy the hypotheses of Lemma~\ref{lem:equicoercivity}, $u_\e \to u$ in $\left[L^2(\OO)\right]^n$.
\end{lemma}
\begin{proof}
Note first that using Lemma~\ref{lem:equicoercivity}, we have that the sequence $\left(u_\e\right)_\e$ is uniformly bounded in $\left[H^1(\OO)\right]^n$ so that there exists $u^* \in V^n$ such that $u_\e \rightharpoonup u^*$ in $\left[H^1(\OO)\right]^n$.
We need to show that $u^*$ satisfies~\eqref{eq:stateSharpWeak} from which we will deduce that $u^*=u$. 

Given any $1 \le j \le n$, let $\psi \in C^0_c(\OO,\mathrm{M}^{d\times d}_{\mathrm{sym}})$ be a test function. 
Denoting $\rho_{i,\e}$,  $1\le i \le m$, the components of $\p_\e$, we have that 
\begin{multline}
    \label{eq:lemma2-1}
    \left|\int_\OO \sum_{i=1}^m\left[a(\p_{i,\e})\mathbb{C}_i\left(\ee(u_{j,\e})-\beta_i s_{j,\e}\mathrm{I}_d\right) - a(\p_i)\mathbb{C}_i\left(\ee(u^*)-\beta_i s\mathrm{I}_d\right)\right]\cdot \psi\, dx\right| \le \\
    \left|\int_\OO\sum_{i=1}^m\left[\left(a(\p_{i,\e})-a(\p_i)\right)\mathbb{C}_i\left(\ee(u_{j,\e})-\beta_i s_{j,\e}\mathrm{I}_d\right)\right]\cdot \psi\, dx\right| \\
    +  \left|\int_\OO \sum_{i=1}^m\left[a(\p_i)\mathbb{C}_i\left(\ee(u_{j,\e})-\ee(u^*) - \beta_i s_{j,\e}\mathrm{I}_d+\beta_i s\mathrm{I}_d\right)\right]\cdot \psi\, dx\right| .
\end{multline}
Since $|\beta_i| \le k_3$,  $u_{j,\e} \rightharpoonup u^*_j$ in $H^1(\OO)$, and $s_{j,\e} \rightharpoonup s_j$ in  $L^2(\Omega)$, the second term in the right-hand side of~\eqref{eq:lemma2-1} converges to 0.
We then write
\begin{multline*}
    \left|\int_\OO\sum_{i=1}^m \left[ \left(a(\p_{i,\e})-a(\p_i)\right)\mathbb{C}_i\left(\ee(u_{j,\e})-\beta_i s_{j,\e}\mathrm{I}_d\right)\right] \cdot \psi\, dx\right|  \\ 
    \shoveright{
    \le \sum_{i=1}^m\left\|\left(a(\p_{i,\e})-a(\p_i)\right)\psi\right\|_{L^2(\OO)}
    \left\|\CCC_i \left(\ee(u_{j,\e})-\beta_i s_{j,\e}\mathrm{I}_d\right) \right\|_{L^2(\OO)}}\\
    \le k_2 \sum_{i=1}^m \left\|\left(a(\p_{i,\e})-a(\p_i)\right)\psi\right\|_{L^2(\OO)}
    \left\|\ee(u_{j,\e})-\beta_i s_{j,\e}\mathrm{I}_d \right\|_{L^2(\OO)},
\end{multline*}
and since $\p_\e \to \p$ in $L^1$ and is uniformly bounded in $L^\infty(\OO)$, we get that $\left\|\left(a(\p_{i,\e})-a(\p_i)\right)\psi\right\|_{L^2(\OO)} \to 0$ for any $1 \le i \le m$.
Using then Lemma~\ref{lem:equicoercivity}, we get that $\ee(u_{j,\e})$ is uniformly bounded in $L^2(\OO)$.
Since $\mathrm{s}_\e$ is uniformly bounded in $L^\infty(\OO)$ hence in $L^2(\OO)$, the first term in the right-hand side of~\eqref{eq:lemma2-1} also converges to 0.
Finally, by density of $C^0_c$ in $L^2$ we get that for any $\phi \in V$, 
$\left|\int_\OO\left(a(\p_{i,\e})-a(\p_i)\right)\mathbb{C}_i\left(\ee(u_{j,\e})-\beta_i s_{j,\e}\mathrm{I}_d\right)\cdot \ee(\phi)\, dx\right|\to 0$ and leveraging~\eqref{eq:stateSharpWeak} that for any $\in V$ and $1 \le i \le m$:
\[
    \int_\OO \sum_{i=1}^m a(\p_i)\mathbb{C}_i\left(\ee(u_j^*)-\beta_i s_j\mathrm{I}_d\right)\cdot \ee(\phi)\, dx = 0,
\]
\emph{i.e.} that $u^* = u$ solves~\eqref{eq:stateSharpWeak}.
\end{proof}

\begin{lemma}[Compactness]
Let $(\p_\e,\mathrm{s}_\e) \in \mathcal{D}_\p \times\mathcal{S}$ and the associated equilibrium displacements $u_{j,\e}$, $1\le j \le n$  be such that $\mathcal{I}(u_{1,\e},\dots,u_{n,\e}) + \alpha\widetilde{\mathcal{P}}_\e(\rho_\e,\mathrm{s}_\e)$ is uniformly bounded.
Then there exists a subsequence $(\p_\e,\mathrm{s}_\e)_{\e'} \subset (\p_\e,\mathrm{s}_\e)_\e$ and $(\p,\mathrm{s}) \in \widetilde{\mathcal{D}}\times\mathcal{S}$ such that $\p_{\e'} \to \p$ in $\left[L^1(\OO)\right]^m$ and $\mathrm{s}_{\e'} \rightharpoonup \mathrm{s}$ in $\left[L^p(\Omega)\right]^n$ for any $1<p\le \infty$.
\end{lemma}

\begin{proof}
    The compactness of the designs $\p_\e$ derives directly from the compactness theorem for $\mathcal{P}_\e$~\cite[Proposition 4.1]{Baldo-1990a}, noting that the set $\mathcal{D}_\rho$ is a closed convex subset of $[H^1(\Omega)]^m$ hence closed under weak convergence whereas that of the stimuli $\mathrm{s}_\e$ derives from the uniform $L^\infty$ bound on stimuli in the definition of $\mathcal S$. 

\end{proof}

\begin{proof}[Proof of Theorem~\ref{thm:main}]
    Having proved the lemmata, the proof of Theorem~\ref{thm:main} is straightforward.
    Note first that from~\cite[Theorem 2.5]{Baldo-1990a}, we have that $\widetilde{\mathcal{P}}_\e \xrightharpoonup{\Gamma(L^1)} \widetilde{\mathcal{P}}$.
    From Lemma~\ref{lem:continuity}, we get that $\widetilde{\mathcal{I}}$ is a continuous function of $(\rho,\mathrm{s})$, so that by stability of $\Gamma$--convergence by continuous perturbations, we get that for any $\alpha >0$ $\widetilde{\mathcal{I}} + \alpha \widetilde{\mathcal{P}}_\e$ $\Gamma$--converges to $\widetilde{\mathcal{I}} + \alpha \widetilde{\mathcal{P}}$ for the $\left[L^1(\OO)\right]^m$ strong times $\left[L^p(\Omega)\right]^n$ weak topology for any $1<p\le \infty$.
    Secondly, from the compactness and continuity lemmas~\ref{lem:equicoercivity} and~\ref{lem:continuity}, and the equi-coercivity and lower semiconinuity of $\mathcal{P}_\e$ (\cite[Proposition 4.1]{Baldo-1990a}), we get that $\widetilde{\mathcal{I}} + \alpha \widetilde{\mathcal{P}}$ admits minimizers for any $\e>0$ and that the minimizing sequence is compact.
    We can then conclude the proof of Theorem~\ref{thm:main} by a direct application of the fundamental theorem of $\Gamma$--convergence.
\end{proof}

\begin{remark}
    Note that the hypotheses of Lemma~\ref{lem:equicoercivity} rule out a degenerate Hooke's law $\mathbb{C}= 0$ for any of the materials (\emph{i.e.} optimizing the distribution of $m-1$ materials with the $m^{th}$ materials playing the role of ``void''). 
    As is common practice, in this situation we  we introduce an artificial stiffness parameter $\eta$ and replace the ``void'' phase with a weak material with Hooke's law $\eta\mathrm{I}$. 
    Taking the limit as $\eta \to 0$ is technical (see for instance~\cite[Section 4]{BourdinChambolle2003}), and we did not attempt to study this limit here.
\end{remark}

\section{Numerical implementation}\label{sec:numerical-implementation}
In all that follows, we focus on the spacial case of three isotropic linear elastic materials, ``void'' (associated with $\rho_1$), a non-responsive material (associated with $\rho_2$), and a responsive material (associated with $\p_3$). 
For the ``void'' and non-responsive materials, we set $\beta_1=\beta_2=0$ in~\eqref{eq:statePFWeak}.
In the responsive material, we set $\beta_3 = 1$.

We handle the constraint $\p_1+\p_2+\p_3=1$ explicitly by substituting $\p_1=1-\p_2-\p_3$ and optimizing with respect to $\tp = (\p_2,\p_3)$ under the constraint $0 \le \p_2,\p_3\le 1$. 
Of course, this means that $\p_1$ only satisfies $-1 \le \p_1 \le 1$.
However, it is easy to see that the proof of Theorem~\ref{thm:main} still holds in this case, provided that $a$ be extended to $[-1,1]$ as an even function.
With an abuse of notation, we write $\mathcal{I}(\tp, \mathrm{s})$ and $\mathcal{P}_\e(\tp,\mathrm{s})$ to denote $\mathcal{I}((1-\rho_2-\rho_3,\rho_2,\rho_3),\mathrm{s})$ and $\mathcal{P}_\e((1-\rho_2-\rho_3,\rho_2,\rho_3),\mathrm{s})$ respectively.

It is then natural to enforce null-stimulus in materials 1 and 2, which is easily achieved by adding a penalty term of the form 
\begin{equation}\label{eq:stimulusPenalty}
    Q(\tp,\mathrm{s}) = \int_{\OO}((1-\rho_2-\rho_3)^2+\p_2^2)\sum_{j=1}^{n}s_j^2\,dx
\end{equation}
to the objective function.

Similarly, instead of enforcing the volume fraction constraints $\int_\Omega \p_i\, dx = \theta_i|\Omega|$, $i=2,3$ strongly, we introduce a penalty term
\begin{equation}\label{eq:volumeContraints}
    V_C(\tp) = \nu_2\int_{\OO}\p_2\;dx + \nu_3\int_{\OO}\p_3\,dx,
\end{equation}
where $\nu_2$ and $\nu_3$ are two penalty factors set by trial and error.


With these changes, our problem becomes
\begin{multline}
    \label{eq:objFunctionNumerics}
    \inf_{(\tp,\mathrm{s})\in \left[H^1(\OO;[0,1])\right]^2 \times \mathcal{S}} \mathcal{O}(\tp,\mathrm{s}) := \widetilde{\mathcal{I}}(\tp,\mathrm{s})
    + \alpha \mathcal{P}_\e(\tp)
    + V_C(\tp) + Q(\tp,\mathrm{s}).
\end{multline}

\subsection{Sensitivity analysis}\label{sec:sensitivity-analysis}
We use the adjoint method~\cite{DelosReyes2015} to compute the sensitivity of the objective to the design.

Given an admissible pair of design variables $(\p,\mathrm{s}) \in \mathcal{D}_\rho \times \mathcal{S}$ and $(u_1,\dots,u_n) \in V^n$ admissible displacement fields, we define the Lagrangian $\LL$
\begin{multline}
     \label{eq:numericalLagrangian}
     \LL(u_1,\ldots,u_n,\tp, \mathrm{s},\lambda_1,\ldots,\lambda_n) =
     \mathcal{O}(\tp,\mathrm{s}) 
     +\sum_{j=1}^n {\sum_{i=1}^3\int_\OO a(\p_i)\mathbb{C}_i\left(\ee(u_j)-\beta_i s_j\mathrm{I}_d\right)\cdot \ee(\lambda_j)\,dx},
\end{multline}
where $(\lambda_1,\dots \lambda_n) \in V^n$ are Lagrange multipliers.

Let  $(u_1(\tp,s_1),\dots,u_n(\tp,s_n))$ be the equilibrium displacements associated to $(\tp,\mathrm{s})$ satisfying~\eqref{eq:statePFWeak} and define
$$ 
    J(\tp,\mathrm{s})=\mathcal{O}\left(u_1(\tp,s_1), \dots, u_n(\tp,s_n)\right).
$$ 
so that 
$$
    \LL(u_1(\tp,s_1), \dots, u_n(\tp,s_n), \tp,s, \lambda_1,\ldots,\lambda_n) = J(\tp, \mathrm{s}).
$$
Formally, the directional derivative of $J$ in a direction $\phi \in \left[H^1(\Omega)\right]^2$ is given by:
\begin{multline}
    \label{eq:frechetDerivative}
    \left \langle \frac{\partial J}{\partial \tp},\phi \right \rangle
     =\left\langle \frac{\partial \LL}{\partial u_1},{\frac{\partial u_1}{\partial \tp}\phi} \right\rangle+\dots+\left\langle \frac{\partial \LL}{\partial u_n},{\frac{\partial u_n}{\partial \tp}\phi} \right\rangle
    + \left\langle \frac{\partial \LL}{\partial \tp},\phi \right\rangle.
\end{multline}
%
As is customary, instead of computing $\partial u_j/\partial \tp$, we choose Lagrange multipliers, $\lambda_1^*,\dots\lambda^*_n$ satisfying the adjoint equations:
\begin{equation}\label{eq:weakFormAdjoint}
    \left\langle \frac{\partial \LL}{\partial u_j},v\right\rangle = \int_{\OO_0}(u_j-\bar{u}_j)v\,dx -\sum_{i=1}^3\int_{\OO}a(\p_i)\mathbb{C}_i \ee(\lambda_j^*)\cdot\ee(v)\,dx=0,\ 1 \le j \le n.
\end{equation}
for any $v \in V$.
With this choice of Lagrange multipliers, we get
\[    
\left \langle \frac{\partial J}{\partial \tp}(\tp,s),\phi \right \rangle = \left\langle \frac{\partial \LL}{\partial \tp}(u_1(\tp,s_1), \dots, u_n(\tp,s_n), \tp, \mathrm{s}, \lambda_1^*, \dots \lambda_n^*),\phi \right\rangle,
\]
\emph{i.e.,}
\begin{multline}
    \label{eq:sensitivityRho}
    \left \langle \frac{\partial J}{\partial \tp}(\tp,s),\phi \right \rangle =  \int_\Omega \alpha \nabla \mathcal{P}_\e(\tp)\cdot \phi +
    \nu_1 \phi_1 + \nu_2\phi_2 \, dx \\
    + \sum_{j=1}^n\sum_{i=1}^3\int_\OO a'(\rho_i)(\mathbb{C}_i\left(\ee(u_j)-\beta_i s_j\mathrm{I}_d\right)\cdot \ee(\lambda_j^*)\phi_i \,dx + \int_\Omega \nabla_{\tp} Q(\tp,\mathrm{s})\cdot \phi\, dx,
\end{multline}
with the convention $\phi_1 =  1-\phi_2 - \phi_3$.

Similarly, for $\psi \in \left[H^1(\Omega)\right]^2,$ the directional derivative of $J$ with respect to $\mathrm{s}$ in the direction $\psi$ is given by:
\begin{multline}
    \label{eq:frechetDerivativeWrtS}
    \left \langle \frac{\partial J}{\partial \mathrm{s}}(\p,\mathrm{s}),\psi \right \rangle
     =\left\langle \frac{\partial \LL}{\partial u_1},{\frac{\partial u_1}{\partial \mathrm{s}}\psi} \right\rangle+\dots+\left\langle \frac{\partial \LL}{\partial u_n},{\frac{\partial u_n}{\partial \mathrm{s}}\psi} \right\rangle
    + \left\langle \frac{\partial \LL}{\partial \mathrm{s}},\psi \right\rangle
\end{multline}
and with the same choice of Lagrange multipliers $\lambda_1^*,\ldots,\lambda_n^*$ satisfying the adjoint equations $(\ref{eq:weakFormAdjoint}),$ we get
\begin{multline}
    \label{eq:sensitivityS}
    \left \langle \frac{\partial J}{\partial \mathrm{s}}(\p,s),\Psi \right \rangle = -\sum_{i=1}^3\int_\OO \left[\mathbb{C}_ia(\p_i)\left(\beta_i\mathrm{I}_2\right)\cdot \ee(\lambda_j^*)\right] \cdot \psi_i\,dx + \nabla_\mathrm{s} Q(q,\mathrm{s})\cdot \psi\, dx.
\end{multline}

At each iteration of the gradient-based minimization algorithm, given $(\tp,\mathrm{s})$, the computation of the sensitivity of our objective function with respect to a design change (resp. with respect to a stimulus change) involves computing the equilibrium displacements $u_1(\tp,\mathrm{s}),\dots u_n(\tp,\mathrm{s})$ by solving $n$ linearized elasticity problems~\eqref{eq:statePFWeak}, then computing the associated adjoint variables $\lambda_1^*, \dots,\lambda_n^*$ using~\eqref{eq:weakFormAdjoint} before evaluating~\eqref{eq:sensitivityRho} (resp. ~\eqref{eq:sensitivityS}.)

\subsection{Minimization with respect to  $\mathrm{s}$}\label{sec:optimal-stimulus}
Observe that, minimizing the objective function $(\ref{eq:objFunctionNumerics})$ is equivalent to minimizing the Lagrangian $(\ref{eq:numericalLagrangian})$ associated with the objective function. For simplicity, we only consider one prescribed displacement $\Bar{u}$ and its corresponding stimulus $s.$ For a fixed design $\p$, the minimization with respect to $s$ is equivalent to 
\begin{spreadlines}{2ex}
\begin{equation}\label{eq:langrangianToStimulus}
\begin{dcases}
\begin{aligned}
      &\min_{s \in \mathcal{S}}\;\; \LL(u,\p,s,\lambda)\text{ defined in }(\ref{eq:numericalLagrangian})\\
     &\text{subject to}\;\;-1 \leq s \leq 1.
\end{aligned}
\end{dcases}
\end{equation}
\end{spreadlines}
By expanding the Lagrangian $(\ref{eq:numericalLagrangian})$ and collect all the terms explicitly depending on $s$, we have
\begin{align*}
    \min_{s \in \mathcal{S}} \LL(u,\rho,s,\lambda) &\iff \min_{s \in \mathcal{S}} -\int_{\OO}s\p_3^2\CCC_3\mathrm{I}_2 \cdot \ee(\lambda)\;dx+\int_{\OO}((1-\p_2-\p_3)^2+\p_2^2)s^2\;dx,\\
    &\iff \min_{s \in \mathcal{S}} -s\p_3^2\CCC_3 \mathrm{I}_2 \cdot \ee(\lambda)+((1-\p_2-\p_3)^2+\p_2^2)s^2\\
    &\iff \min_{s \in \mathcal{S}} -s\p_3^2\CCC_3\ee(\lambda)\cdot \mathrm{I}_2+((1-\p_2-\p_3)^2+\p_2^2)s^2\\
    &\iff \min_{s \in \mathcal{S}} -s\p_3^2d\kappa_3\text{tr}(\ee(\lambda))+((1-\p_2-\p_3)^2+\p_2^2)s^2,
\end{align*}
where $\kappa_3 := \frac{d\lambda_3^{**} +2\mu_3^{**}}{d}$ is the bulk modulus of the responsive material. The expression above is quadratic in $s$ in the form of $As+Bs^2$ and its minimizer $s^{*}$ is given as: if $A = 0,$ then $s^{*} = 0,$ and if $2B < |A|$ then $s^{*} = 1$ (resp.  $-1$) if  $\text{tr}(\ee(\lambda))> 0$  (resp. $\text{tr}(\ee(\lambda))< 0$). Otherwise $s^{*} = A/2B$.

\begin{remark}
As $\varepsilon \to 0,$ one can see, if $\p_3 = 0$ at any point $x \in \OO$ \emph{i.e.} we have no responsive material at point $x$, then the optimal stimulus at point $x$ becomes $0.$ And if $\p_3 = 1$ any point $x \in \OO$ \emph{i.e.} we have responsive material at point $x$, then the optimal stimulus becomes either $1$ or $-1$ depending on the sign of $\text{tr}(\ee(\lambda)).$
and the closed form of an optimal stimulus is given as:
\begin{equation}\label{eq:optimaltimulus}
   s^{*}= \begin{cases} 
      1 & \text{if tr}(\ee(\lambda))> 0\text{ and }\p_3=1\\
      0 & \text{if}\;\p_3=0,\\
      -1 & \text{if tr}(\ee(\lambda))< 0\text{ and }\p_3=1.
   \end{cases}
\end{equation}
\end{remark}

\subsection{Minimization algorithm}
\label{sec:minAlgo}

Our implementation uses Firedrake~\cite{FiredrakeUserManual}, an open source automated system for the solution of partial differential equations using the finite element method.
The displacement, density, and stimulus fields are discretized using linear Lagrange simplicial finite elements over structured meshes.
The minimization algorithm we used is a BNCG (Bounded Nonlinear Conjugate Gradient) solver implemented in the TAO (Toolkit for Advanced Optimization) optimization package, which is a part of PETSc (the Portable, Extensible Toolkit for Scientific Computation) library~\cite{petsc-efficient,petsc-user-ref,petsc-web-page}.
BNCG only requires first order derivatives of the objective function, which we computed in close form (\emph{i.e.} optimize then discretize)  in Section~\ref{sec:sensitivity-analysis}.

We tested two different numerical approaches. 
In a \emph{monolithic scheme}, we jointly minimize $\mathcal{O}$ with respect to $\tp$ and $\mathrm{s}$ simultaneously using~\eqref{eq:sensitivityRho} and~\eqref{eq:sensitivityS}.
In a \emph{staggered scheme}, we use TAO to minimize $\mathcal{O}$ with respect to $\tp$ only. 
Whenever computing $\frac{\partial J}{\partial \tp}$ using~\eqref{eq:sensitivityRho}, 
we perform a full minimization of $\mathcal{O}$ with respect to $\mathrm{s}$ using~\eqref{eq:optimaltimulus} after computing the state and adjoint variables $u_j$ and $\lambda_j$ ($1 \le j \le n$).

In each case, we leverage TAO's line search and convergence criteria.

\section{Numerical results}\label{sec:numerical-results}
We present a series of numerical simulations illustrating the strengths of our approach.
In all that follows, we use
\begin{equation}
    \label{eq:defW}
    W(\p) := \p_1^2(1-\p_1)^2+\p_2^2(1-\p_2)^2+\p_3^2(1-\p_3)^2
\end{equation}
as the multi-well potential in $\mathcal{P}_\e$~\eqref{eq:defPe}.
Note that although this function admits 8 roots in $\mathbb{R}^3$, once restricted to the hyperplane $\rho_1 + \rho_2 + \rho_3=1$, it admits only three roots and is therefore admissible in~\eqref{eq:defPe}.
Note that $W$ does not satisfy~\eqref{eq:normalizationW} as such.
A simple rescaling of $W$ and $\alpha$ would be sufficient to do so. 
We use a simple quadratic material interpolation function $a(s) := s^2$.

\subsection{Cantilever beam}
We start with a simple problem with only one prescribed displacement $\bar{u},$ \emph{i.e.} $n=1$ in~\eqref{eq:objectiveFunctionNoPerimeter}. The design domain  $\OO = (0,L_x) \times (0,L_y)$ with $L_x=1$ and $L_y=1/3$.
We prescribe null displacement on the left side $\Gamma_D$ of $\Omega$ and set $\OO_0= (L_x-a,L_x) \times(L_y/2-a/2,L_y/2+a/2)$ with  $a=1/15$.

The prescribed displacement on $\OO_0$ is taken as $\bar{u}=[0,1]^T$ \emph{i.e.} we want the region $\OO_0$ to move upward. The domain $\OO$ is discretized with a structured mesh with cell size $h = \num{2e-3}$. 
The regularization parameter is $\e = \num{2e-3}$ and the perimeter penalization parameter is $\alpha = \num{6e-4}$.
The relative and absolute tolerance  on the gradient and  objective function in TAO were set to $\num{1e-6}$ as it was observed that tighter tolerances do not lead to significant differences in the designs produced.

We first compare the numerical approaches described above.
Figure~\ref{fig:schemeComparison} shows the density of non-responsive ($\rho_2$) and responsive ($\rho_3$) materials, as well as a composite plot showing the non-responsive material in black and the responsive material coloured according to the value of the stimulus $s$ with $s = 1$ in red, $s = 0$ in white, and $s = -1$ in blue. 
In both case, $\rho_2$ and $\rho_3$ are initialized with a constant value 0.3. 
The penalty terms are set to $\nu_2 = 0.1$ and $\nu_3 = 0.3$, and both materials are isotropic homogeneous with non-dimensional Young's modulus 5 and Poisson ratio 1.

We observe that the methods lead to different but well-defined designs, where the material densities are well focussed near 0 and 1, as expected in the phase-field approach.
Upon convergence, the volume fractions of responsive and non-responsive materials are 13.2\% and 15.7\% for the monolithic approach and 12.2\% and 17.5\% for the staggered scheme.
The monolithic approach converged in 380 iterations of TAO's BNCG solver. The final value of the objective function is $\mathcal{O} = \num{4.49e-1}$.
The staggered solver converged in just 173 iteration leading to an objective function $\mathcal{O}= \num{4.17e-3}$.
Considering the non-convexity of the problem, it is not expected that either algorithm will provide a global minimizing design, hence the fact that both approaches lead to significantly different designs is not surprising.
Interestingly, though, the mechanism for the deformation of the two structures is very different. While the  monolithic approach leads to a truss-like structure with deforming cross-members, the geometry of the design obtained by the staggered scheme is more reminiscent of a bi-metal strip where the outer layers deform and the shear stiffness of the central area is maximized.   

All results presented further in this article were computed using the staggered scheme.

\begin{figure}[h!]
    \centering
    \includegraphics[width=0.45\textwidth]{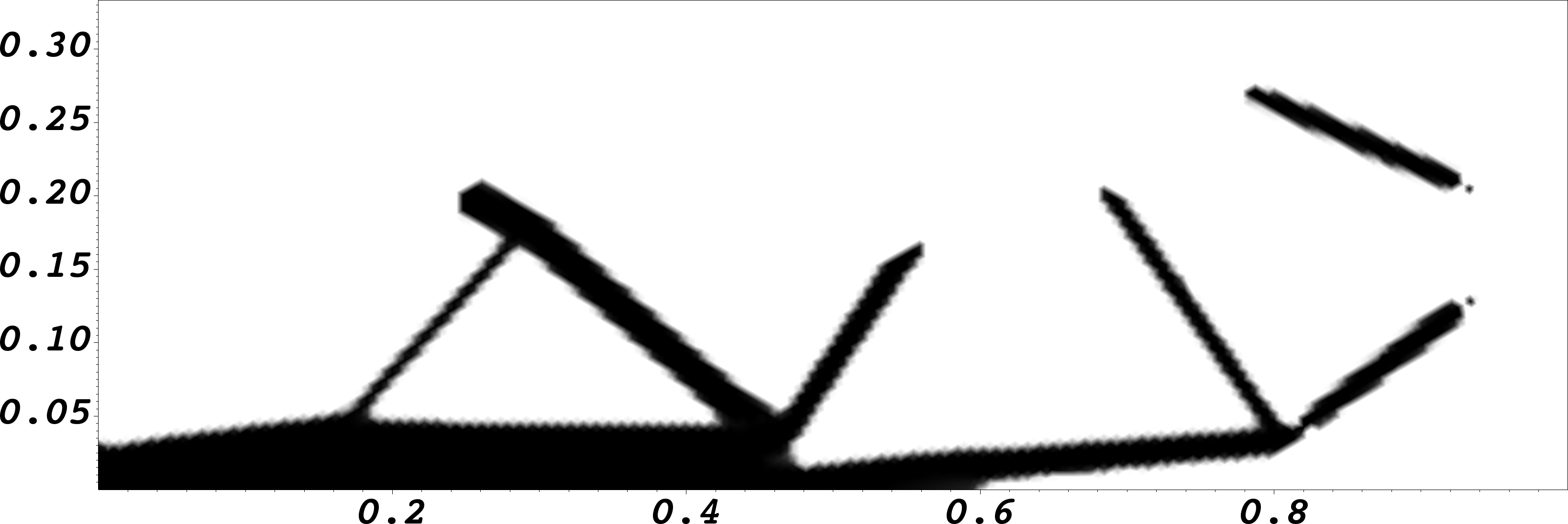}
    \hspace*{.05\textwidth}
    \includegraphics[width=0.45\textwidth]{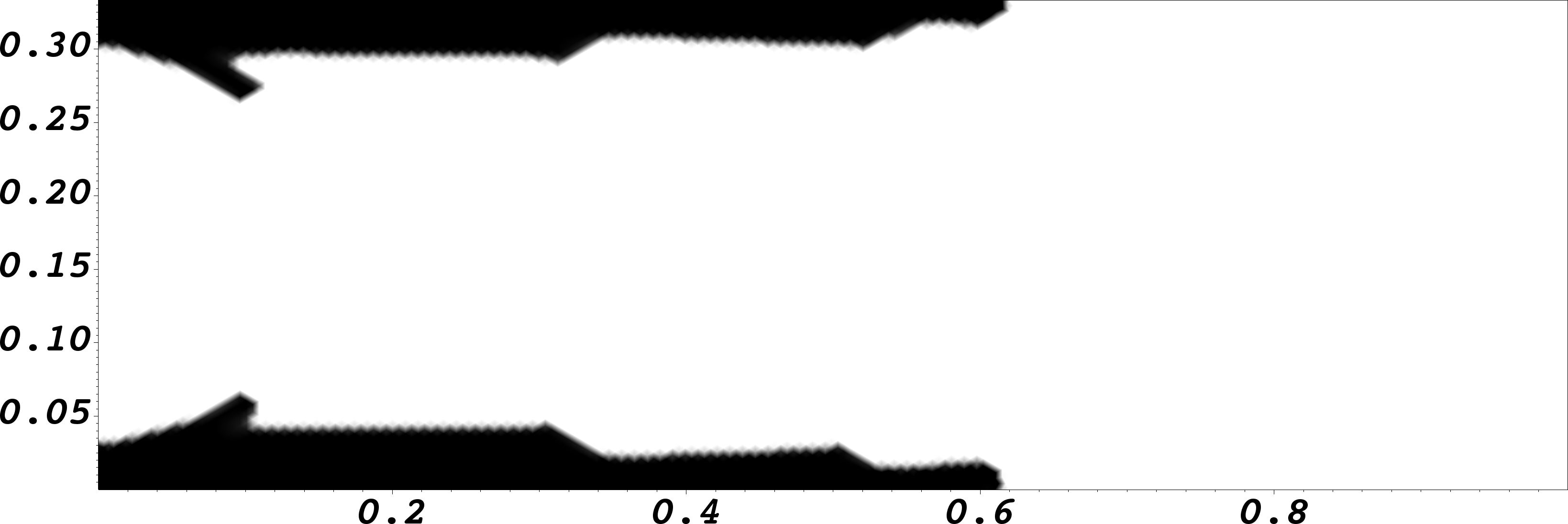}
    \includegraphics[width=0.45\textwidth]{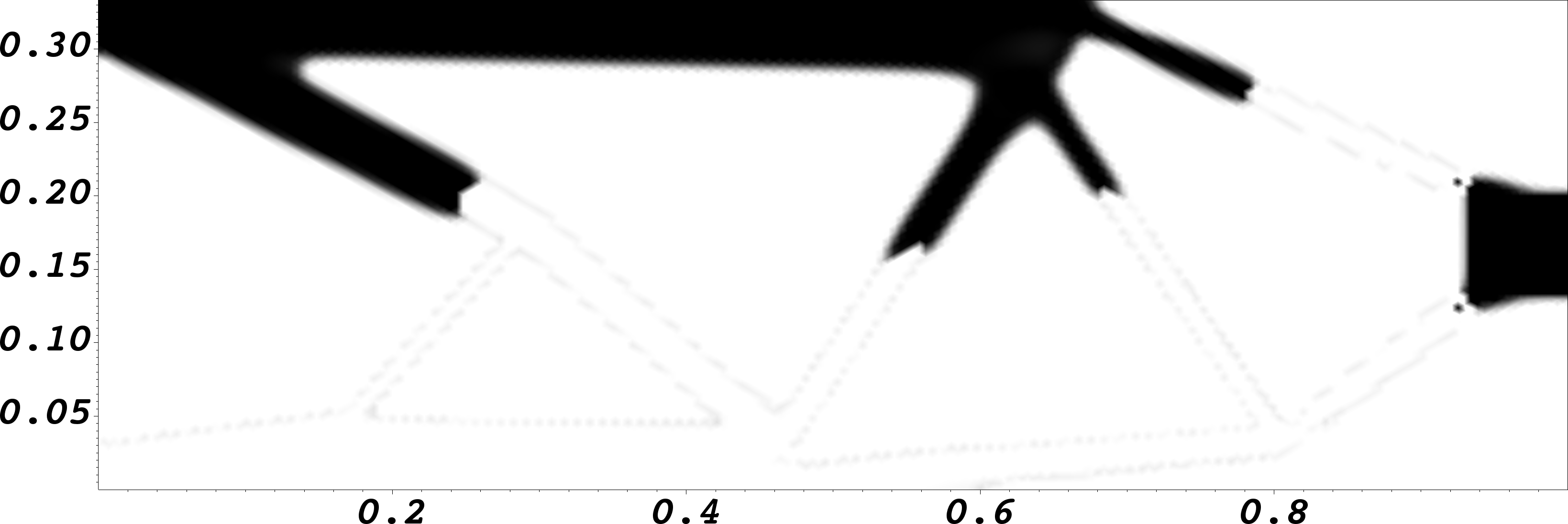}
    \hspace*{.05\textwidth}
    \includegraphics[width=0.45\textwidth]{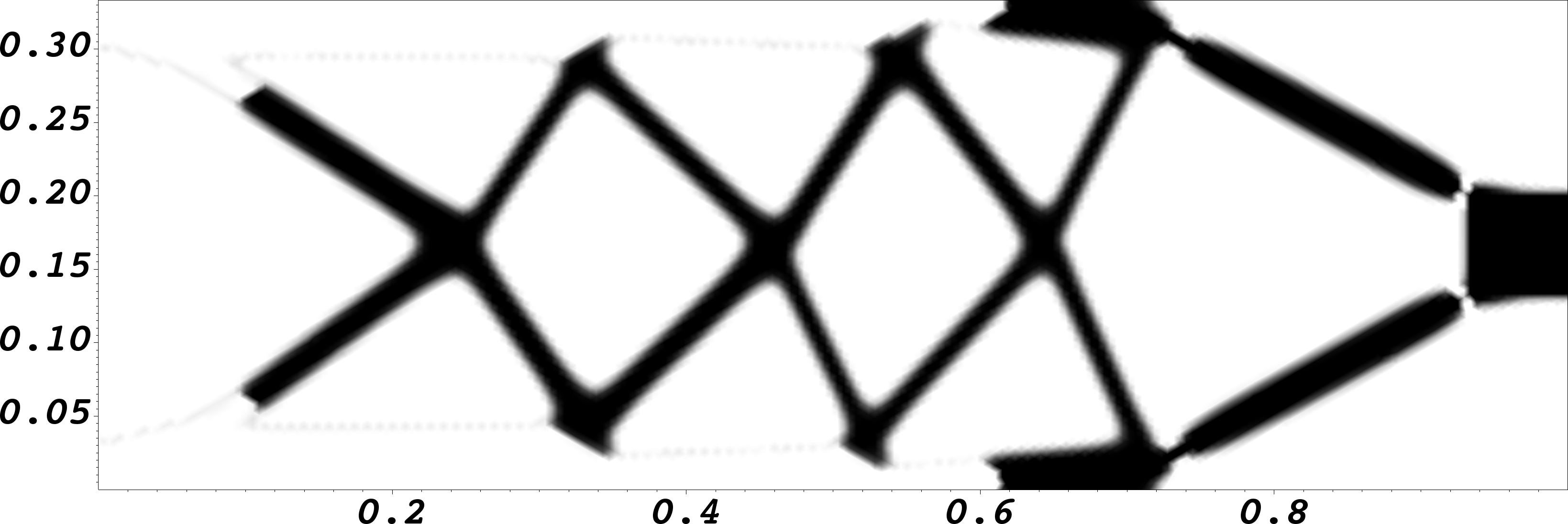}
    \includegraphics[width=0.45\textwidth]{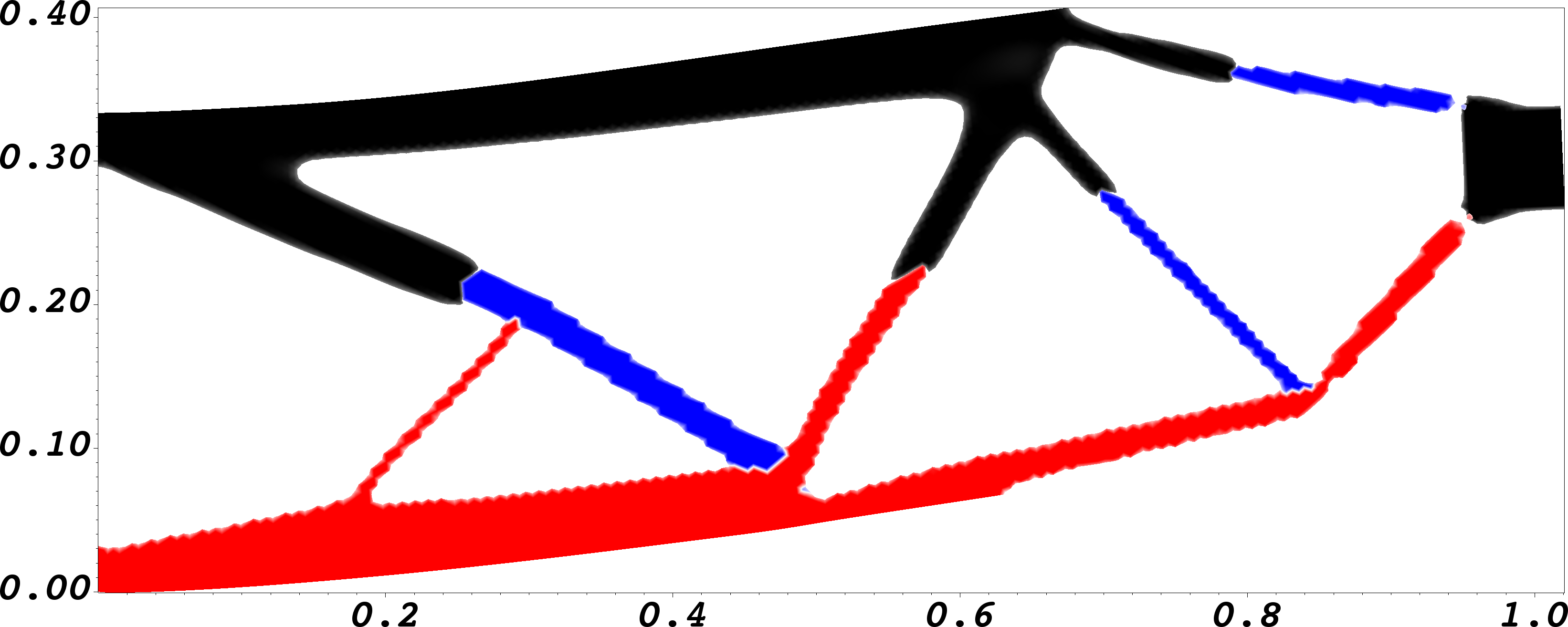}
    \hspace*{.05\textwidth}
    \includegraphics[width=0.45\textwidth]{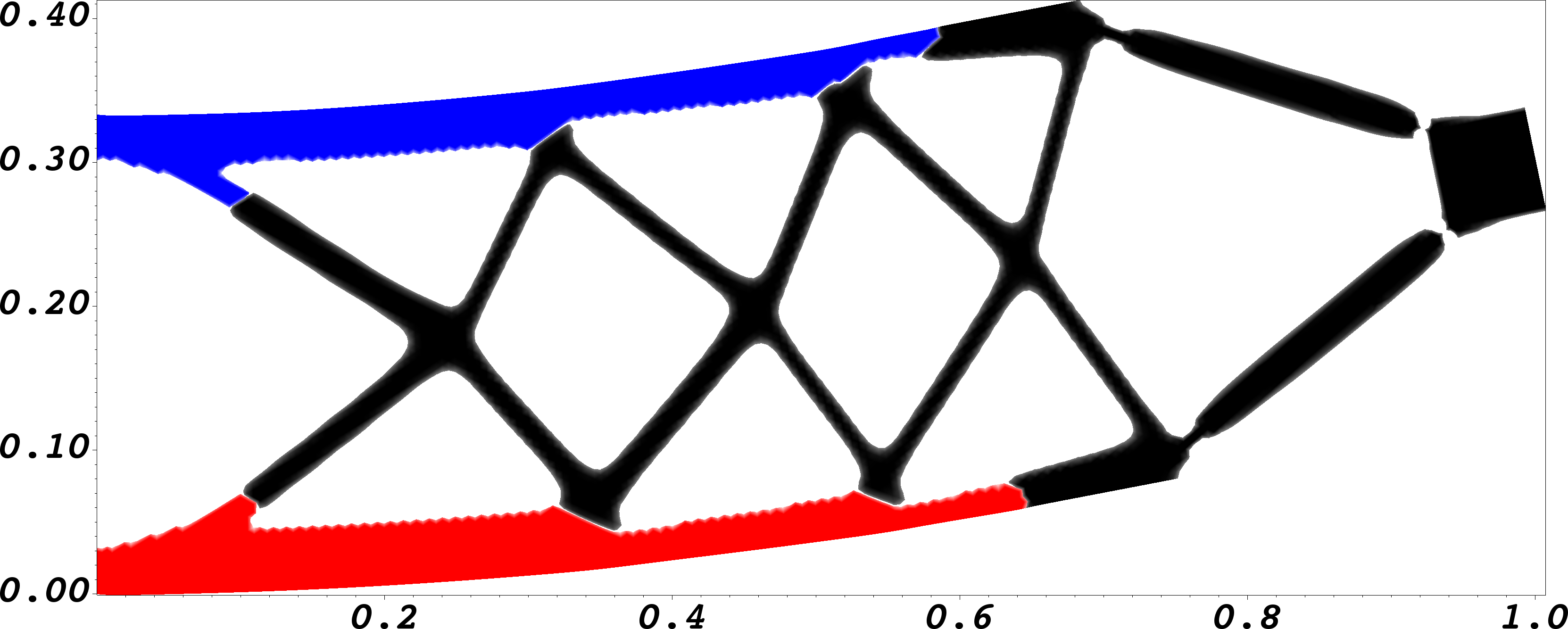}
    \caption{Monolithic (left) vs. staggered (right) scheme. Responsive material density (top), non-responsive material density (middle), and composite plot of both material density and the stimulus in the deformed configuration.}
    \label{fig:schemeComparison}
\end{figure}

Changing the ratio of the Young's modulus of the two materials leads to very different designs.
In Figure~\ref{fig:BeamRatio2}, the Young's modulus of the non-responsive material has been increased to 10.
The penalty terms are respectively $\nu_2 = 0.24$ and $\nu_3 = 0.12$ (left) and $\nu_2 = 0.18$ and $\nu_3 = 0.12$ (right).
When the penalty term on the stiffer, non-responsive material is high enough, the structure consists entirely of the weaker responsive material, whereas decreasing this parameter leads back to rigid truss-like structures activated by small regions of responsive material.

\begin{figure}[h!]
    \centering
    \includegraphics[width=0.45\textwidth]{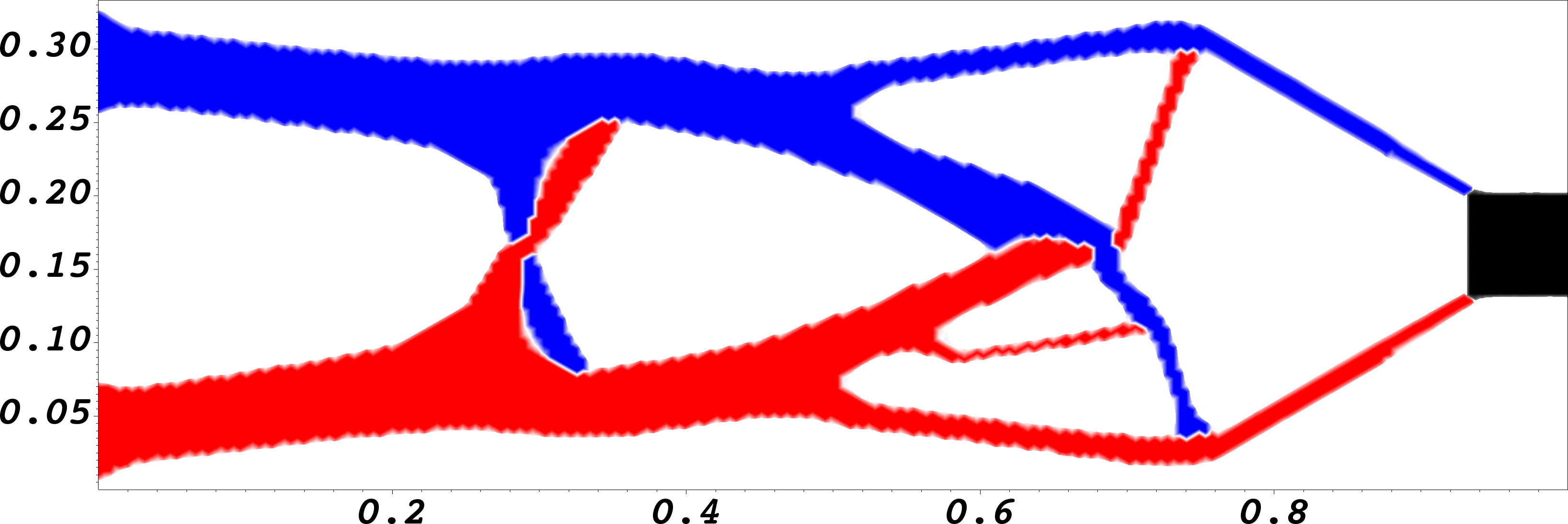}
    \hspace*{.05\textwidth}
    \includegraphics[width=0.45\textwidth]{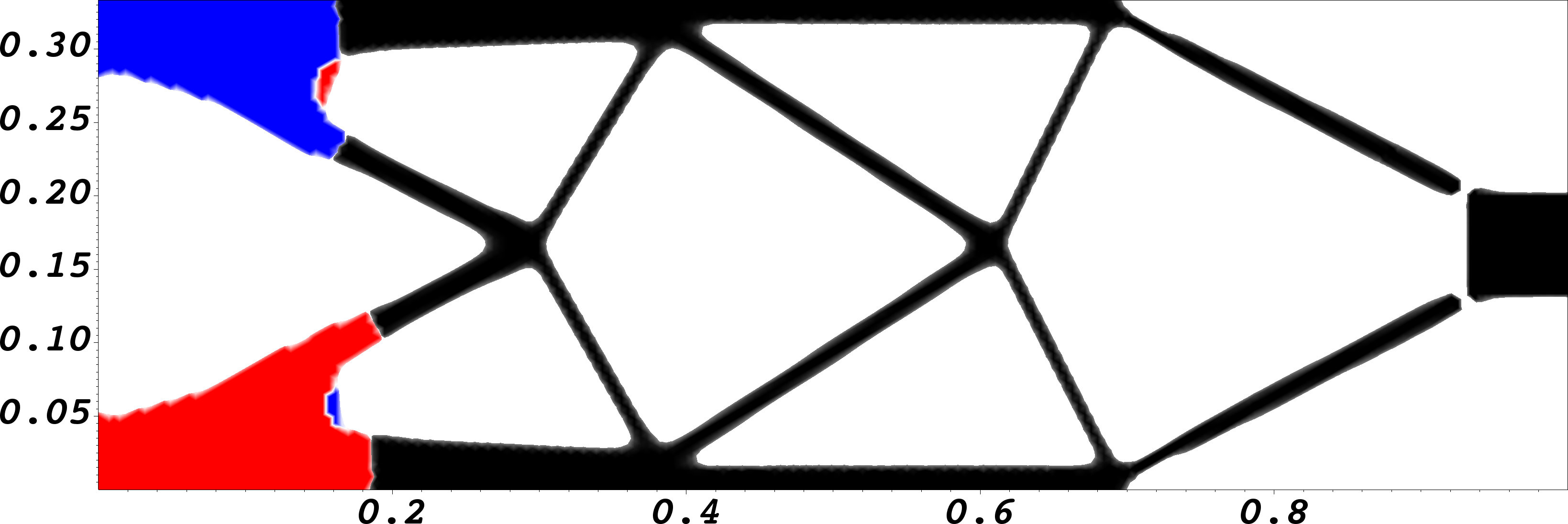}
    \includegraphics[width=0.45\textwidth]{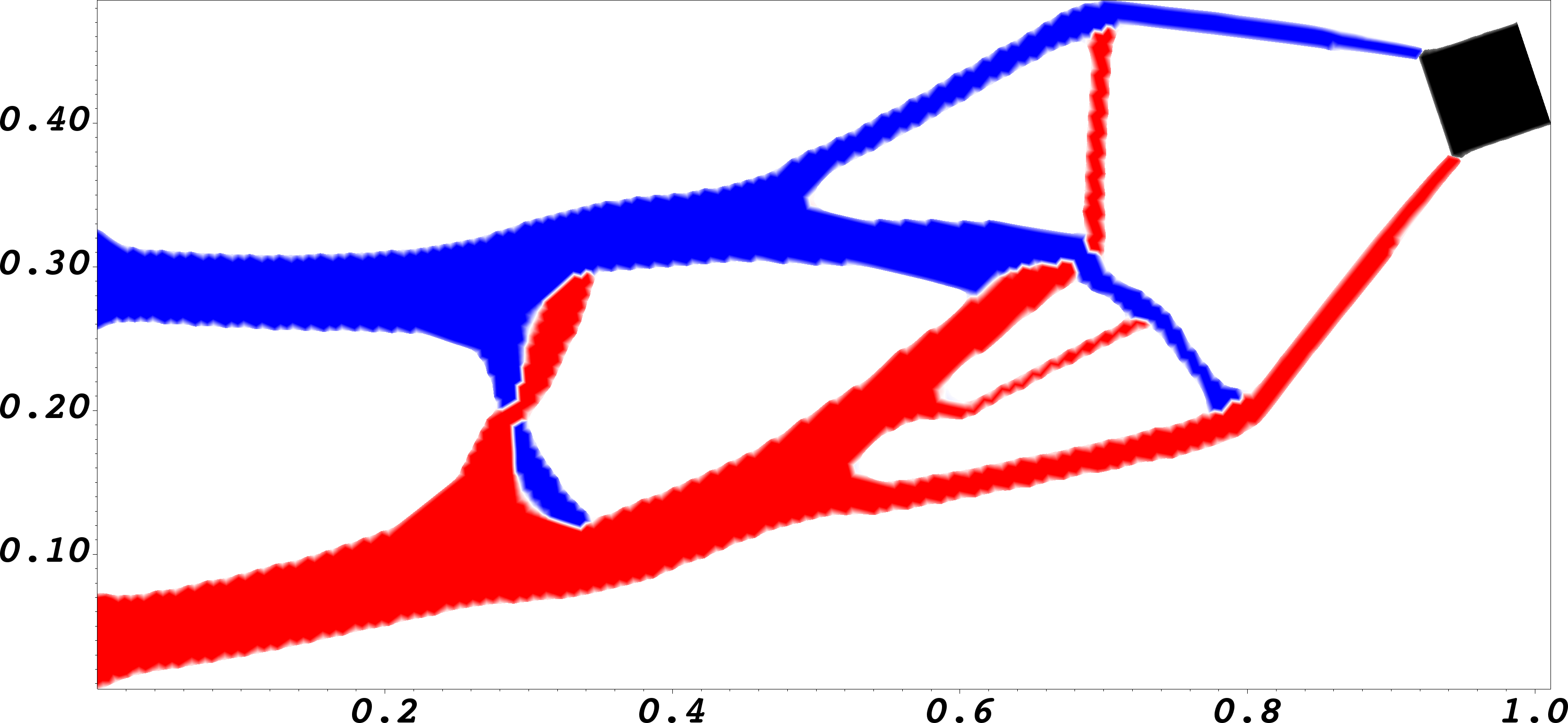}
    \hspace*{.05\textwidth}
    \includegraphics[width=0.45\textwidth]{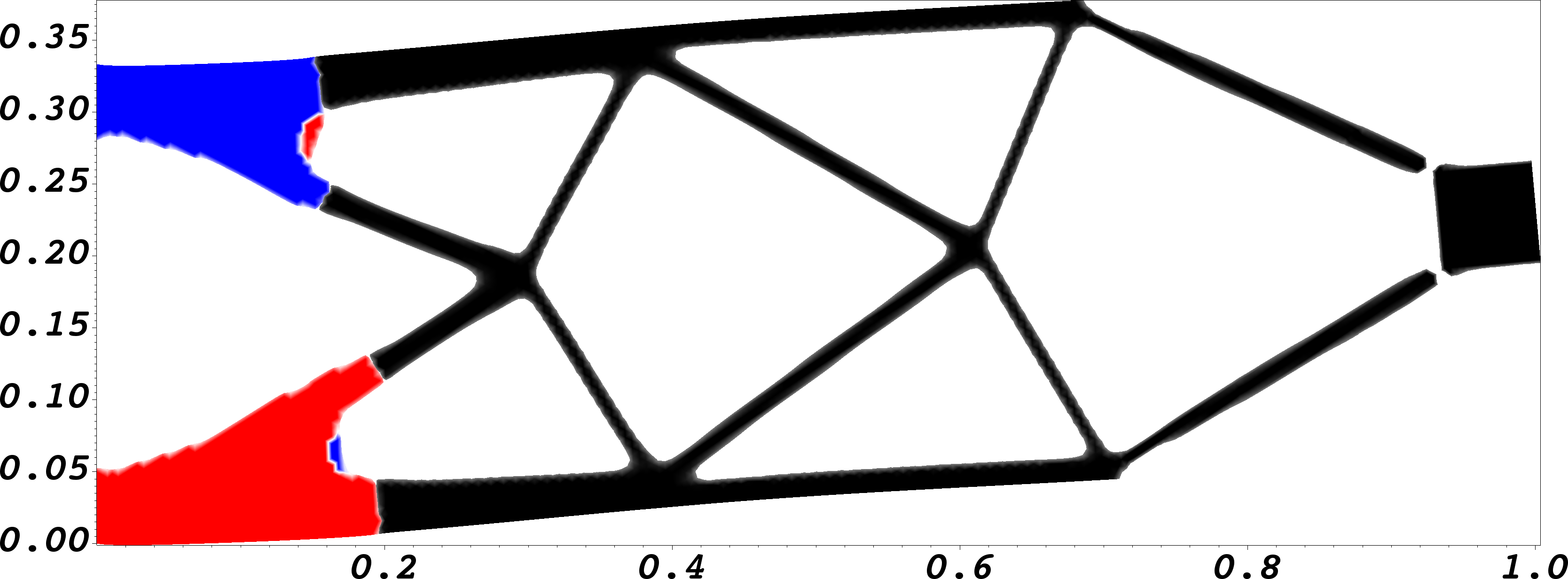}
    \caption{Optimized structure with a ratio $E_2/E_3 = 2$ in the reference (top) and deformed  configuration (bottom) with $\nu_2 = 0.24$ and $\nu_3 = 0.12$ (left) and $\nu_2 = 0.18$ and $\nu_3 = 0.12$ (right).}
    \label{fig:BeamRatio2}
\end{figure}

\subsection{Cantilever beam with two target displacements}
In a second numerical example, we consider two target displacements ($n=2$ in~\eqref{eq:objectiveFunctionNoPerimeter}).
The young's modulus of both materials is set to 5.
In Figure~\ref{fig:Beam2LoadsI}, the target displacements are $(0,1)$ and $(0,2)$.
The penalty terms are $\nu_2 = 0.5$ and $\nu_3 = 0.7$, leading to volume fraction of the responsive and non-responsive material of respectively 15\% and 21\%. 
As the materials are linear elastic, approaching the target displacement $(0,1)$ could have been obtained by rescaling the stimulus in of target displacement $(0,2)$.
Instead, our scheme generates a more complex geometry and activation scheme.
\begin{figure}[h!]
    \centering  
    \includegraphics[width=.45\textwidth]{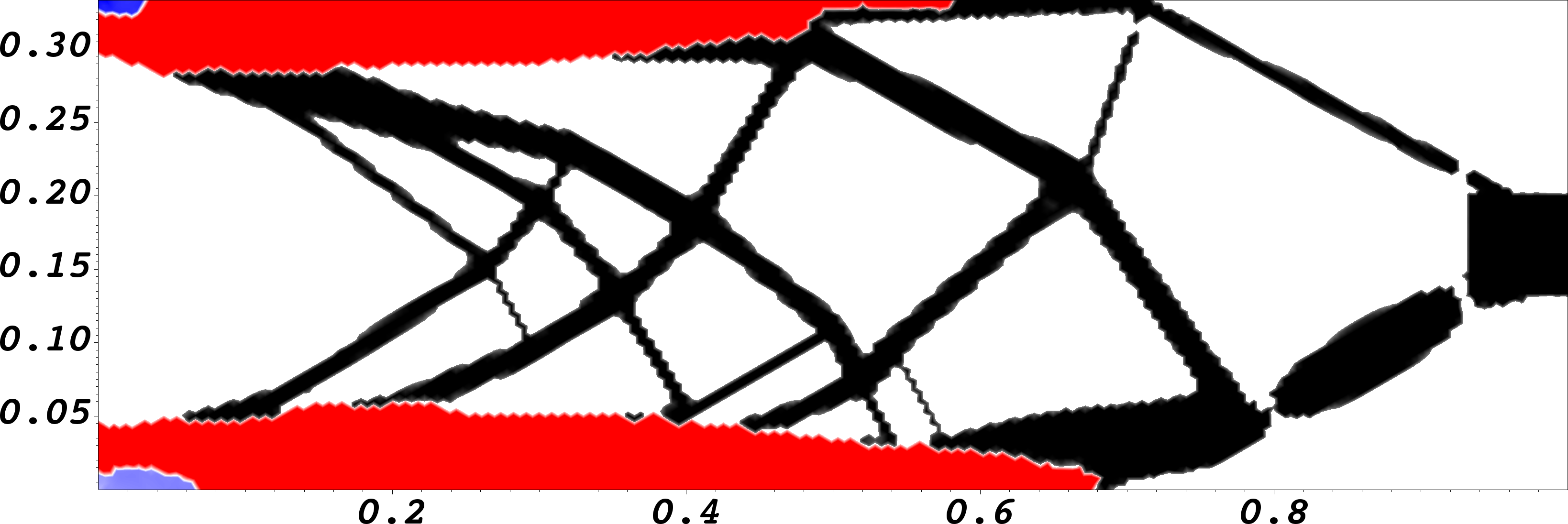}    
    \includegraphics[width=.45\textwidth]{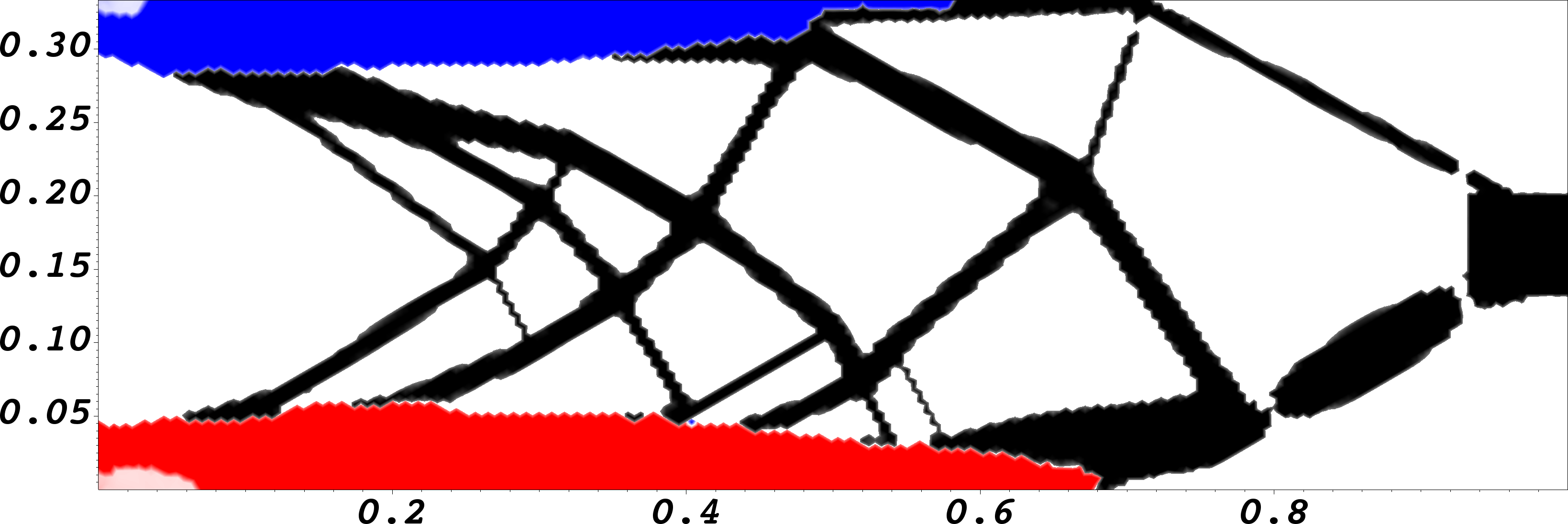}    
    \includegraphics[width=.45\textwidth]{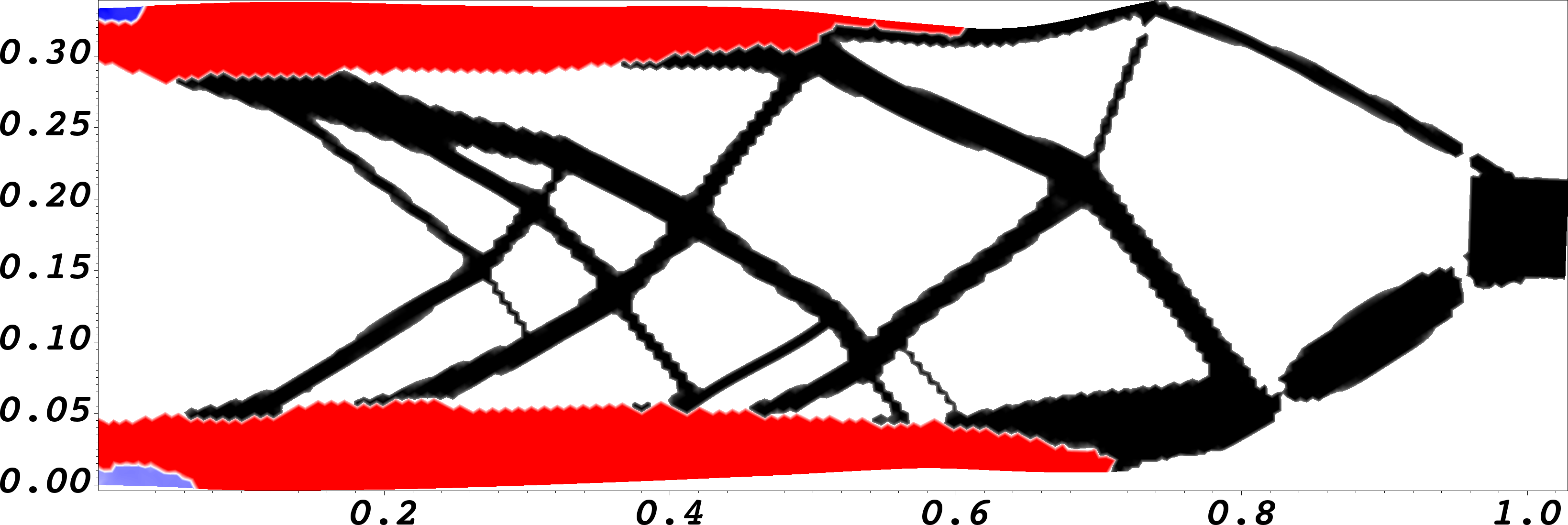}    
    \includegraphics[width=.45\textwidth]{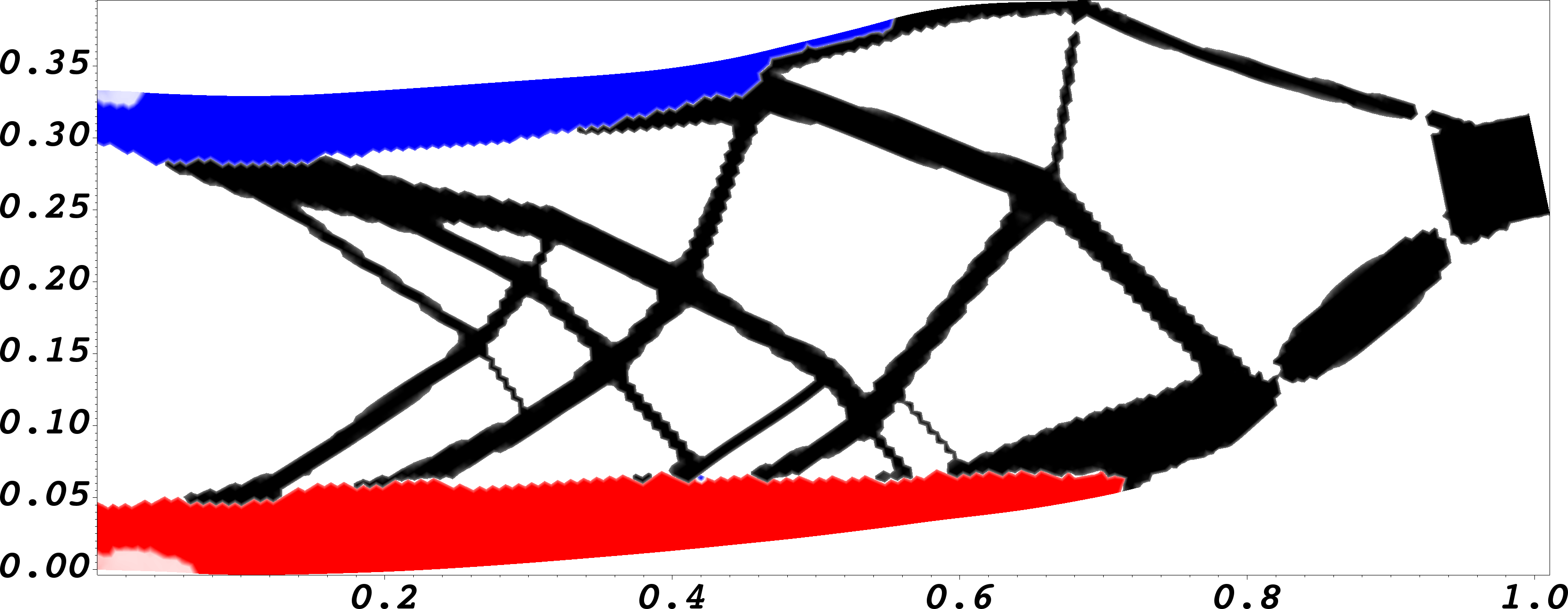}    
    \caption{Optimal design of a beam with two target displacements. (left) Material distribution and stimulus for a target displacement of $(0,1)$ in the reference (top) and deformed (bottom) configuration. (right) Material distribution and stimulus for a target displacement of $(0,2)$ in the reference (top) and deformed (bottom) configuration.}
    \label{fig:Beam2LoadsI}
\end{figure}

In Figure~\ref{fig:Beam2LoadsL}, the target displacements are $(1,0)$ and $(0,1)$.
All other parameters remain the same as in Figure~\ref{fig:Beam2LoadsI}.
The volume fraction of responsive and non-responsive material are respectively 12\% and 9\%.
The responsive material is laid out in simple regions while the  non-responsive material layout forms a stiff truss structure.

\begin{figure}[h!]
    \centering  
    \includegraphics[width=.45\textwidth]{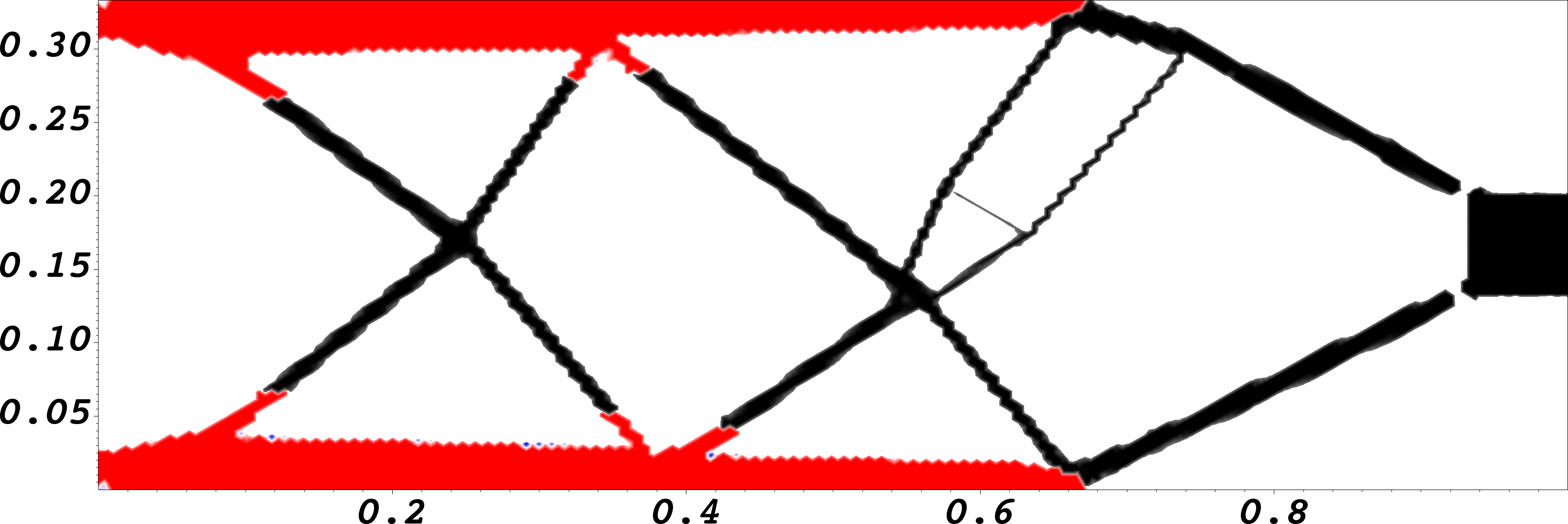}    
    \includegraphics[width=.45\textwidth]{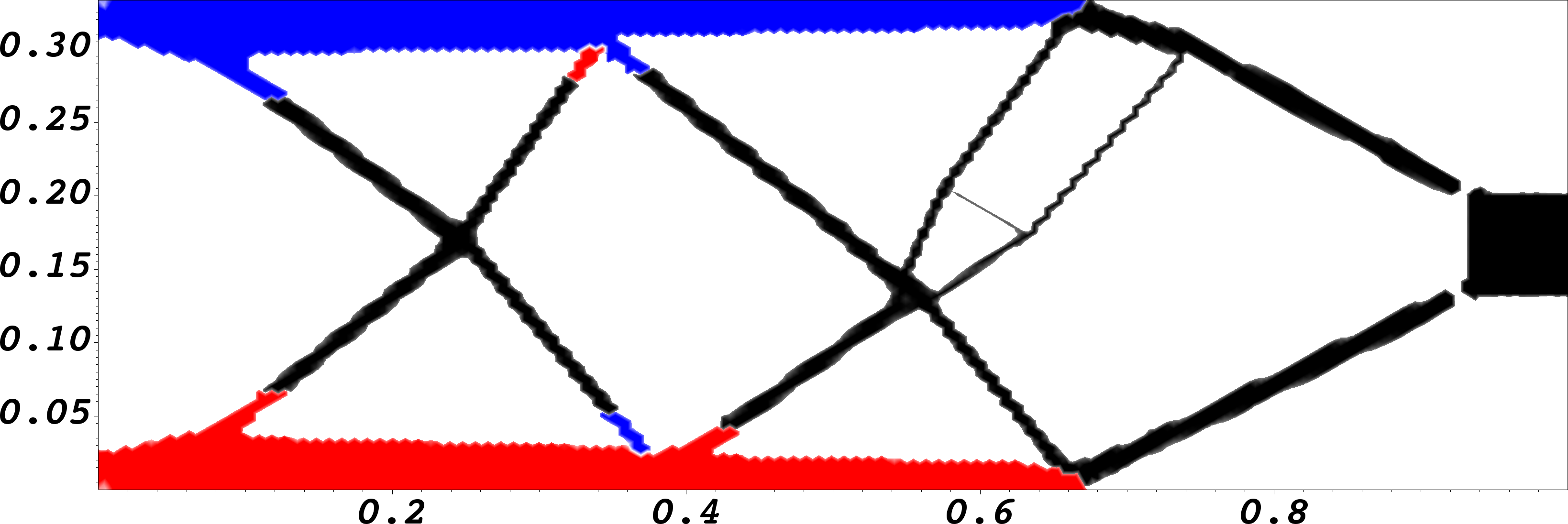}    
    \includegraphics[width=.45\textwidth]{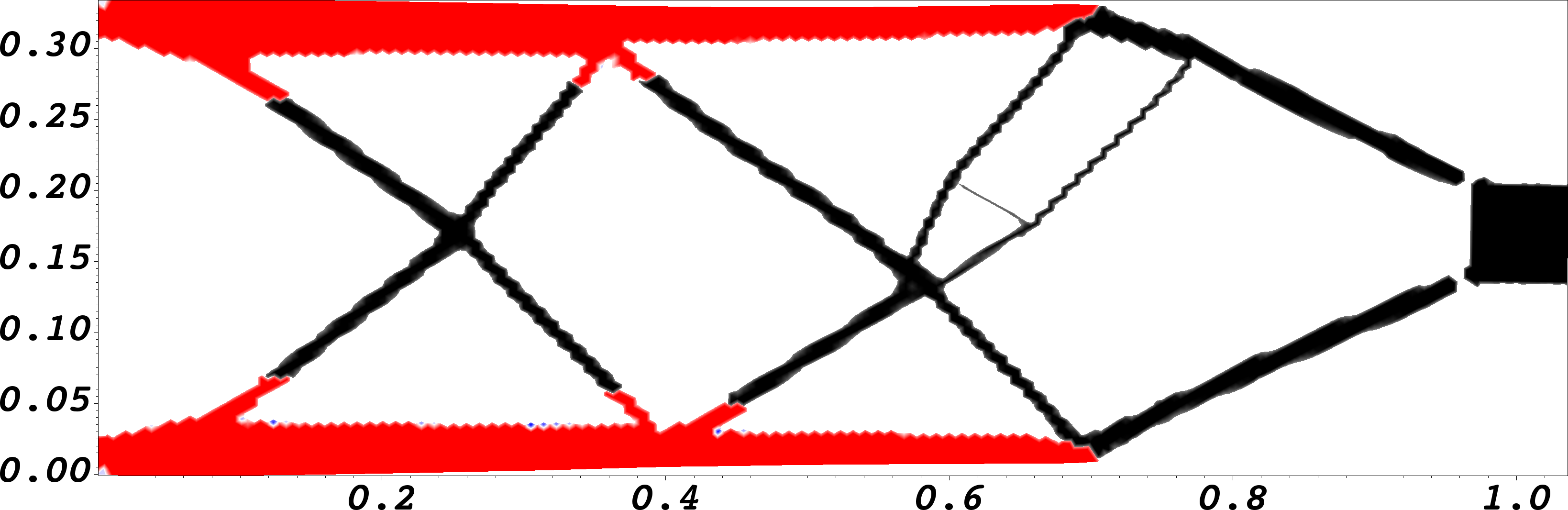}    
    \includegraphics[width=.45\textwidth]{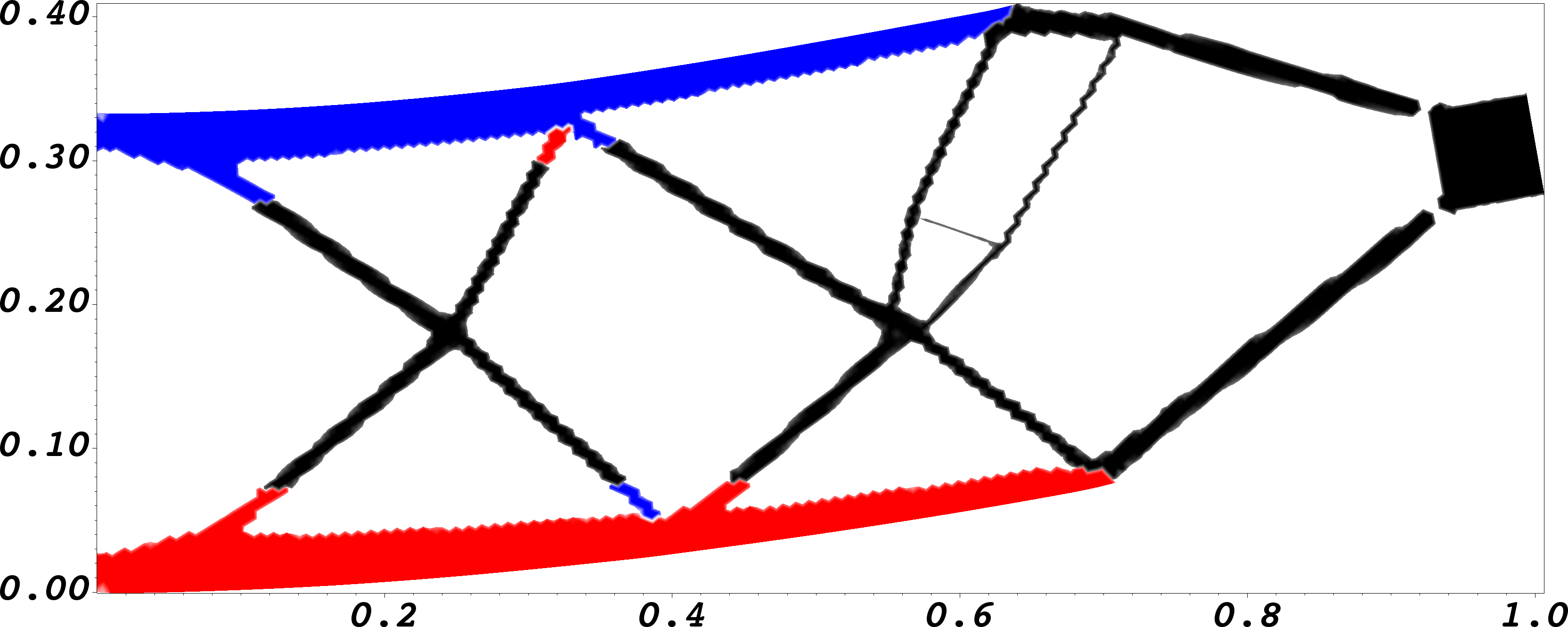}    
    \caption{Optimal design of a beam with two target displacements. (left) Material distribution and stimulus for a target displacement of $(1,0)$ in the reference (top) and deformed (bottom) configuration. (right) Material distribution and stimulus for a target displacement of $(0,1)$ in the reference (top) and deformed (bottom) configuration.}
    \label{fig:Beam2LoadsL}
\end{figure}

\subsection{Hexagonal domain with three prescribed displacements}

Our third example is inspired by the Stewart platform parallel manipulator.
We consider a regular hexagonal domain with edge length 0.35 clamped on three non-consecutive edges (see Figure~\ref{fig:hexagonalDomain}). 
The target displacements of a centered regular hexagon with edge length 0.035 are $\bar{u}_1=(\cos{(0)},\sin{(0)}), \bar{u}_2=(-\cos{(\pi/3)},\sin{(\pi/3)})$ and $\bar{u}_3=(-\cos{(\pi/3)},-\sin{(\pi/3)}).$ 

\begin{figure}[h!]
    \centering
    \includegraphics[width=0.3\textwidth]{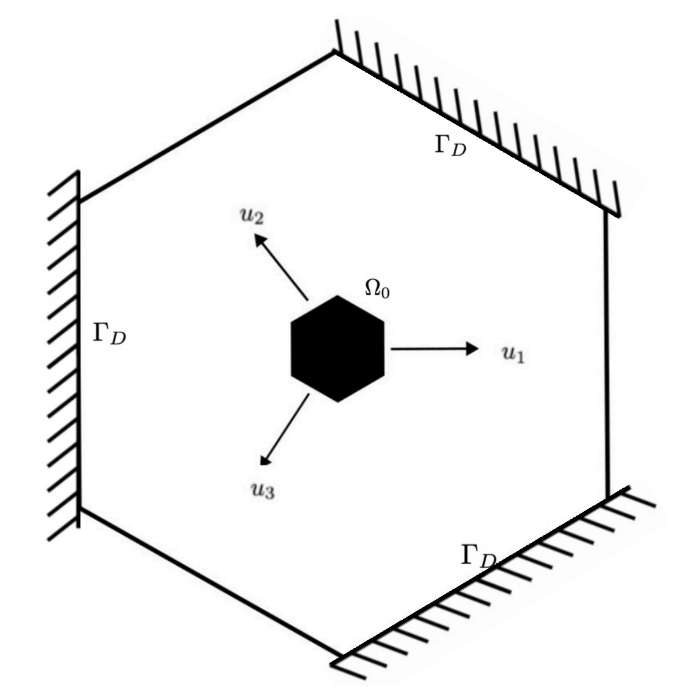}
    \caption{Hexagonal domain clamped at three sides $\Gamma_D$}
    \label{fig:hexagonalDomain}
\end{figure}

In Figure~\ref{fig:Hexagon-Ratio01-08}, the Young's modulus of the responsive and non-responsive materials are set to \num{5e-3} and \num{5e-2} respectively. 
The penalty factor on the stiffer material is set to a much higher value than that of the responsive materials ($\nu_2 = 0.7$ and $\nu_3 = 0.03$).
The perimeter penalty factor is set to \num{3.5e-4} and the regularization length to \num{2e-3}, as above. 
This leads to slender structures activated by large ``pads''.
We note that although that designs and stimuli are invariant by a $2\pi/3$ rotational symmetry, which was not enforces in the computations.

\begin{figure}[h!]
    \centering
    \includegraphics[width=0.3\textwidth]{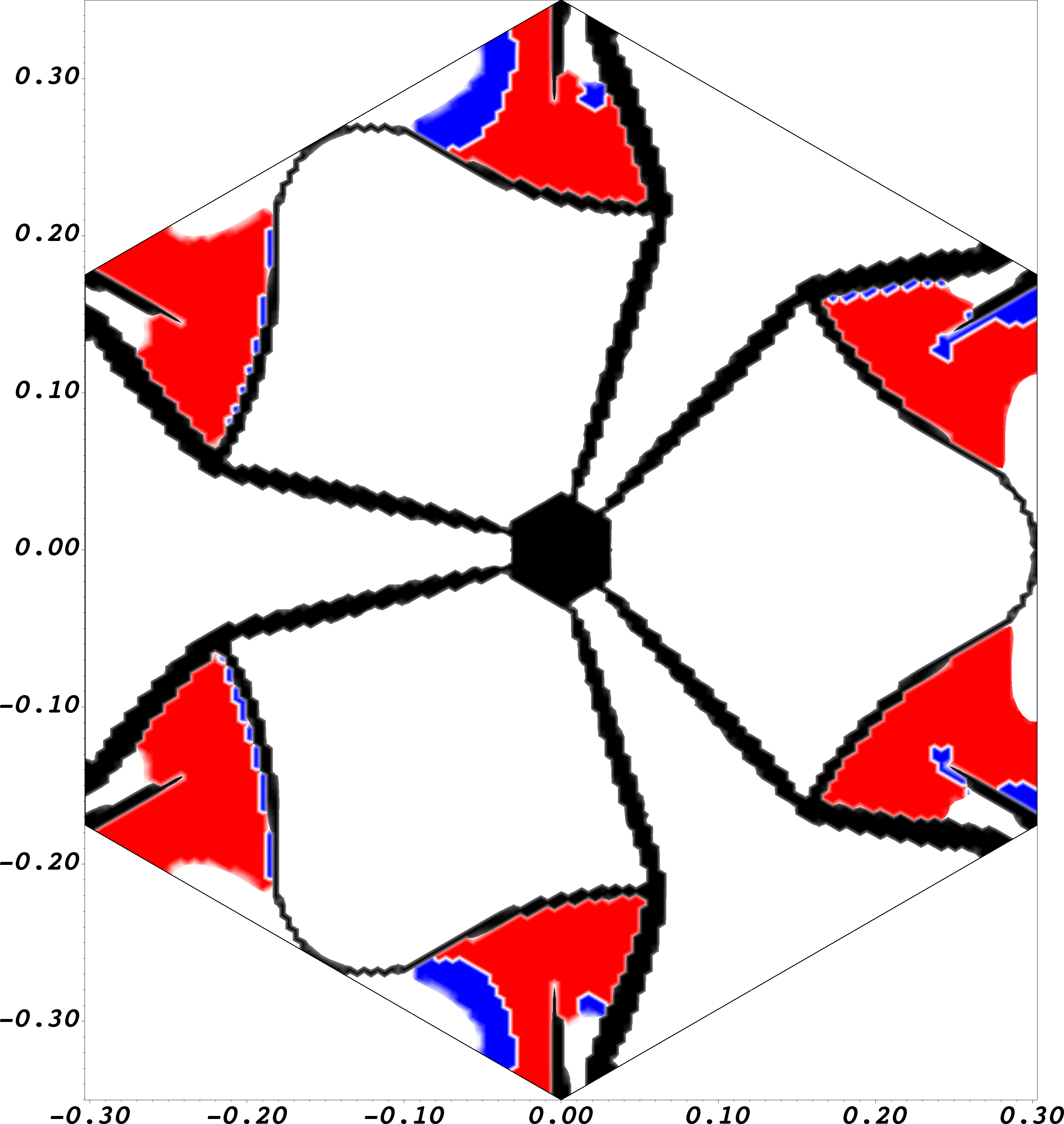}
    \includegraphics[width=0.3\textwidth]{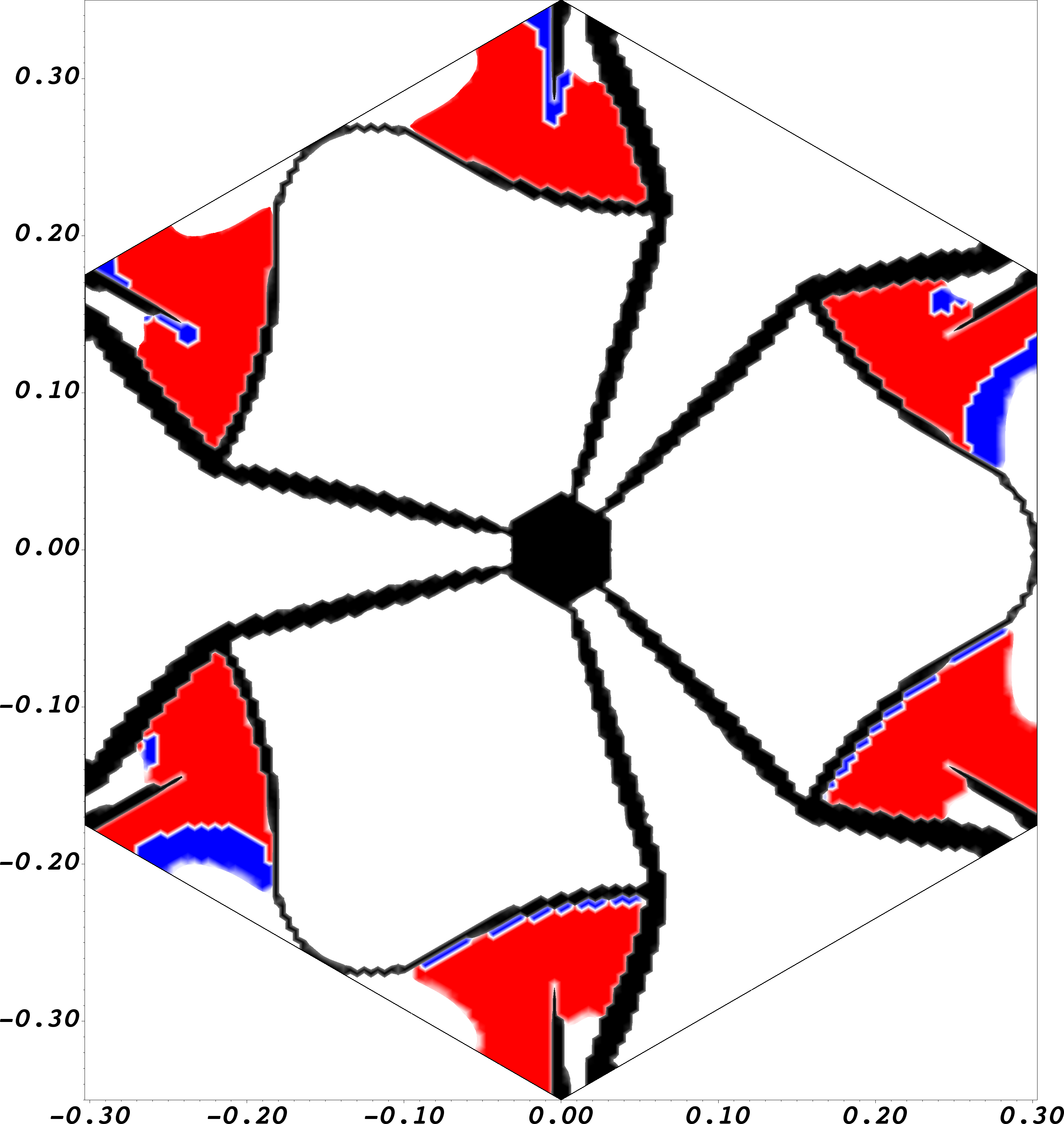}
    \includegraphics[width=0.3\textwidth]{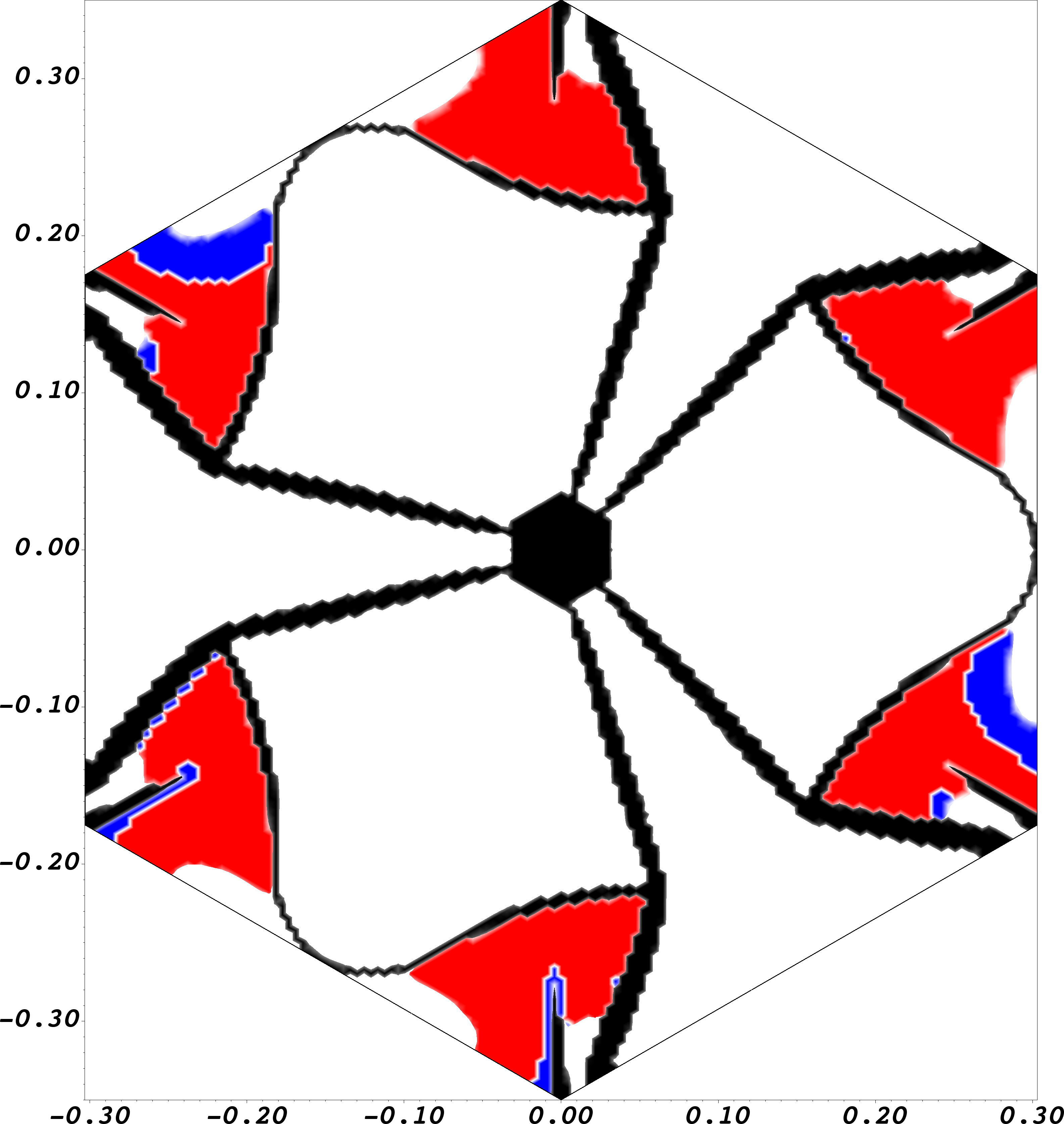}
    \includegraphics[width=0.3\textwidth]{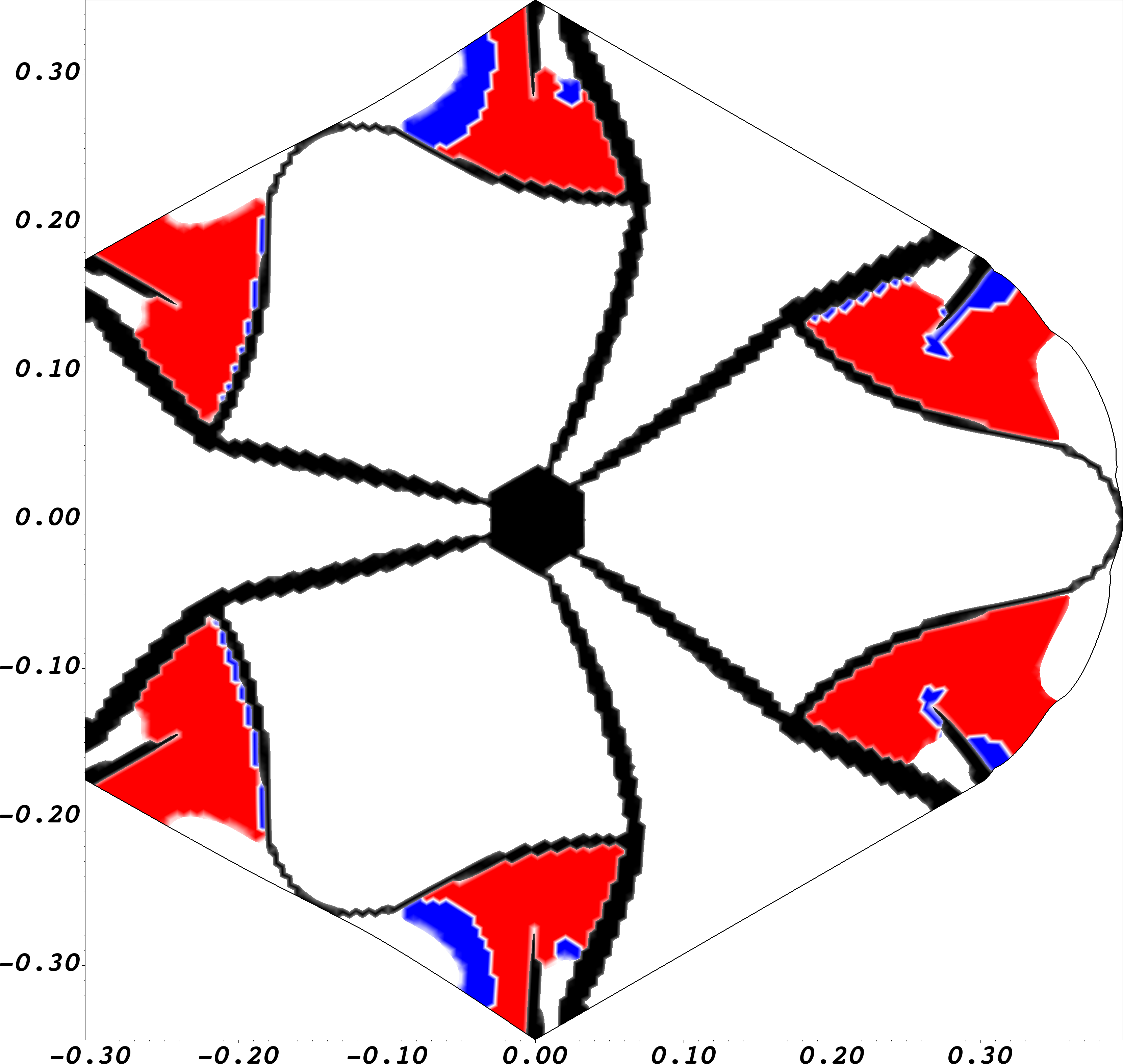}
    \includegraphics[width=0.3\textwidth]{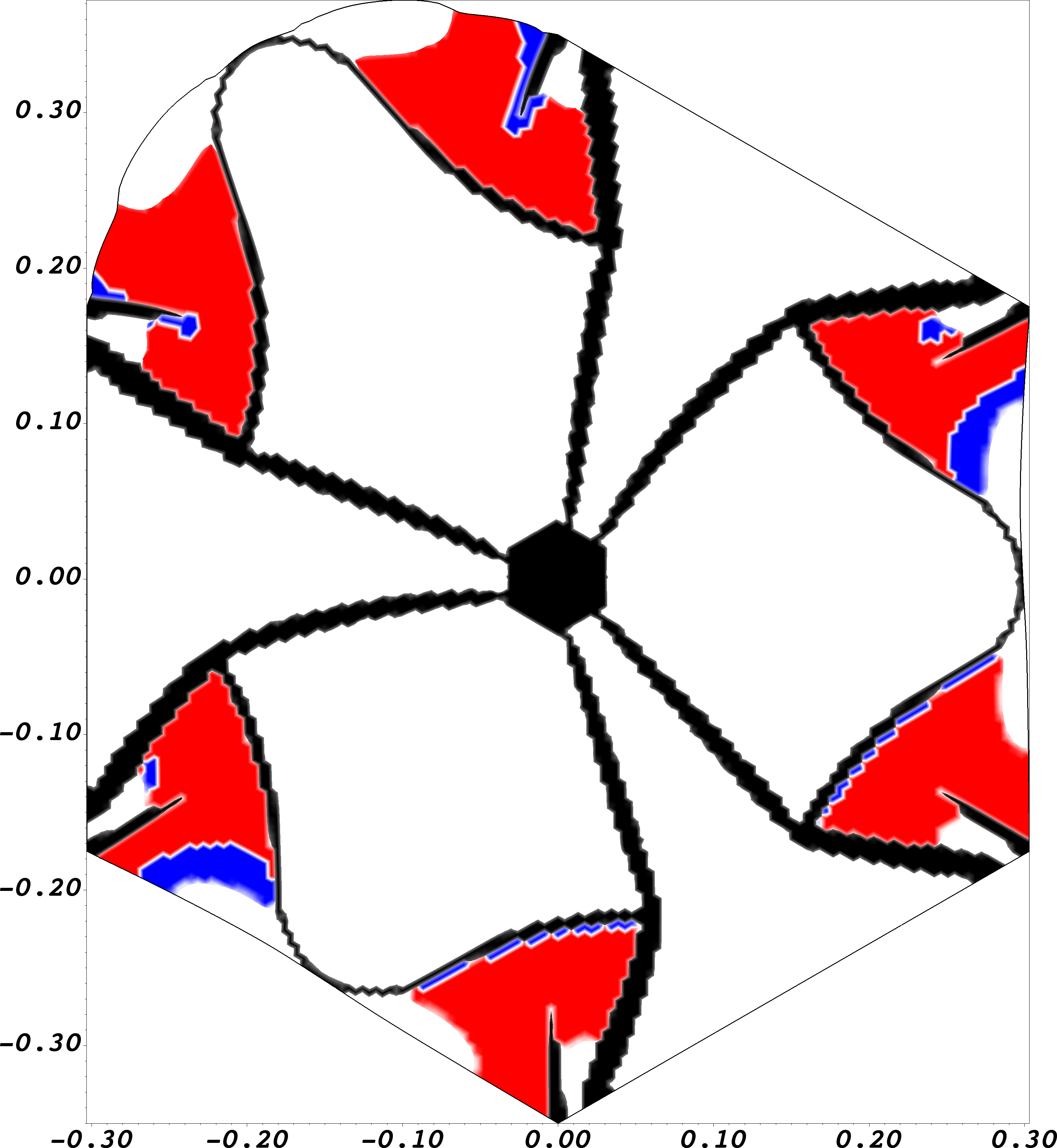}
    \includegraphics[width=0.3\textwidth]{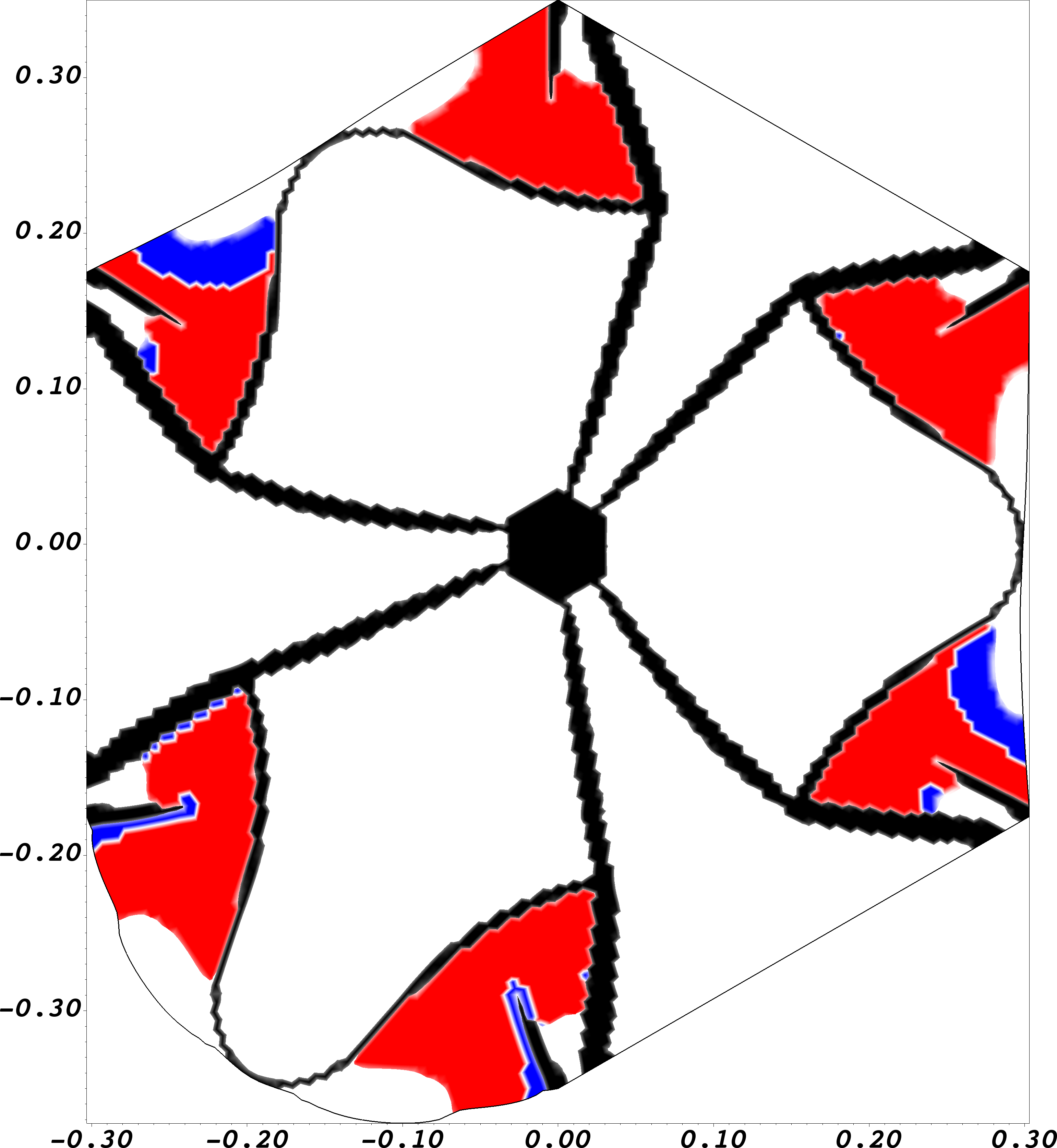}
    \caption{Optimal design with three prescribed displacement with $E_2/E_3 = 10$ and a strong penalty on the stiff material.}
     \label{fig:Hexagon-Ratio01-08}
\end{figure}

When reducing the elastic contrast between materials ($E_2 = \num{5e-2}$ and $E_3 = \num{1e-2}$)  and slightly decreasing the cost of the stiff material ($\nu_2 = 0.3$ and $\nu_3 = 0.03$), we obtain a simpler geometry with larger areas occupied by responsive material.
Again, deformation towards the target displacement is achieved by flexing elongated stiff structures (see Figure~\ref{fig:Hexagon-Ratio05-04}).

\begin{figure}[h!]
    \centering
    \includegraphics[width=0.3\textwidth]{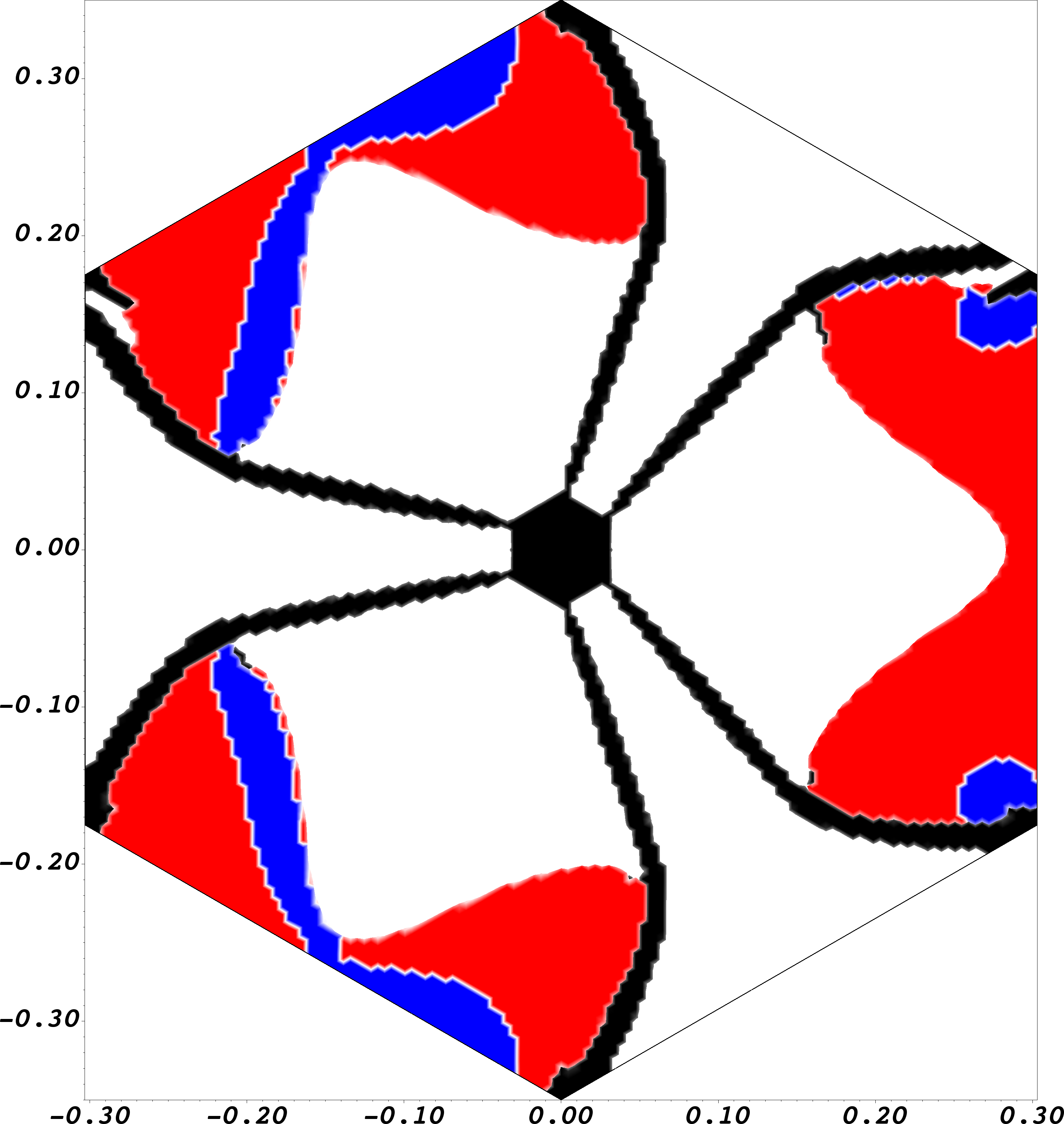}
    \includegraphics[width=0.3\textwidth]{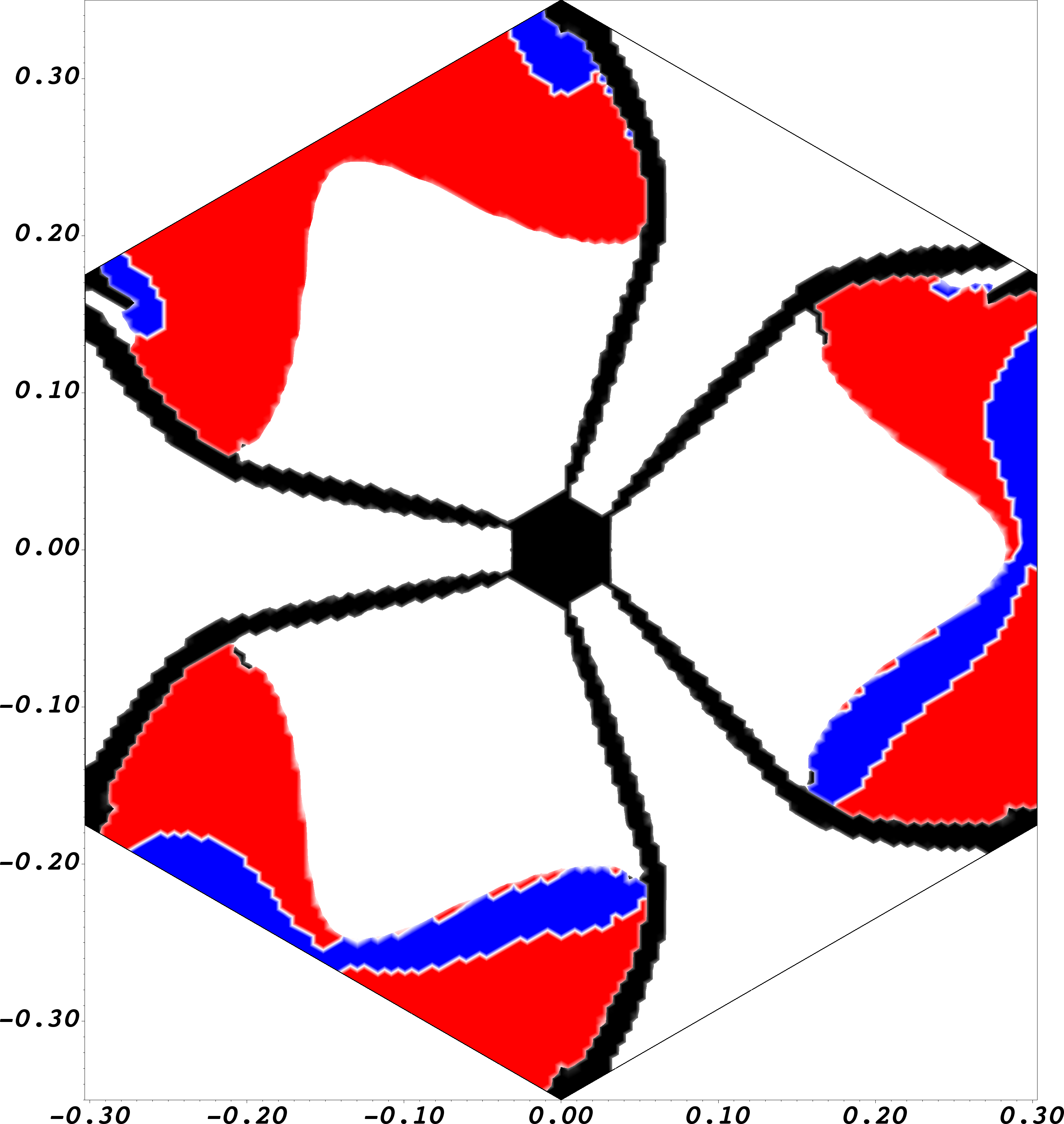}
    \includegraphics[width=0.3\textwidth]{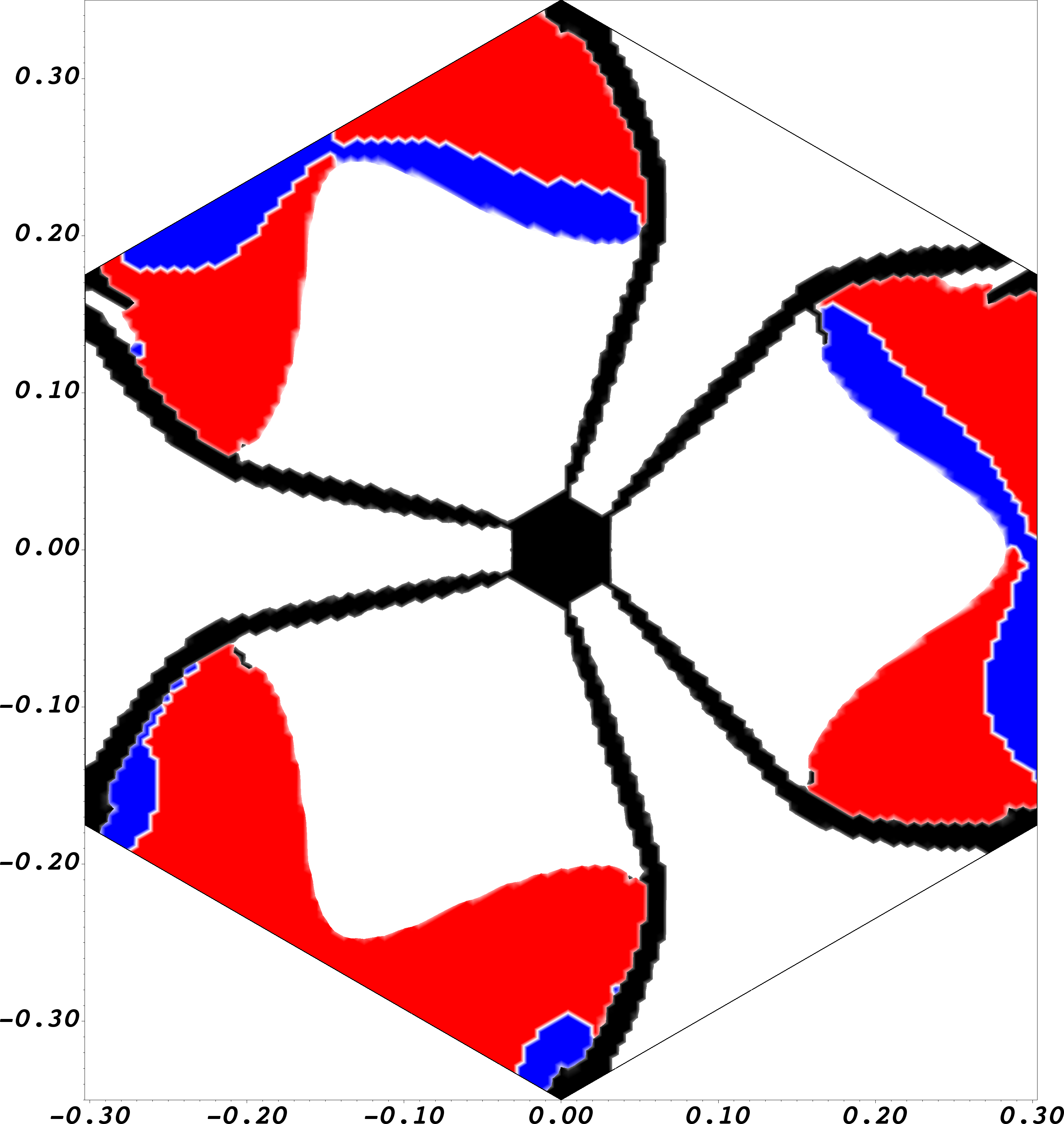}
    \includegraphics[width=0.3\textwidth]{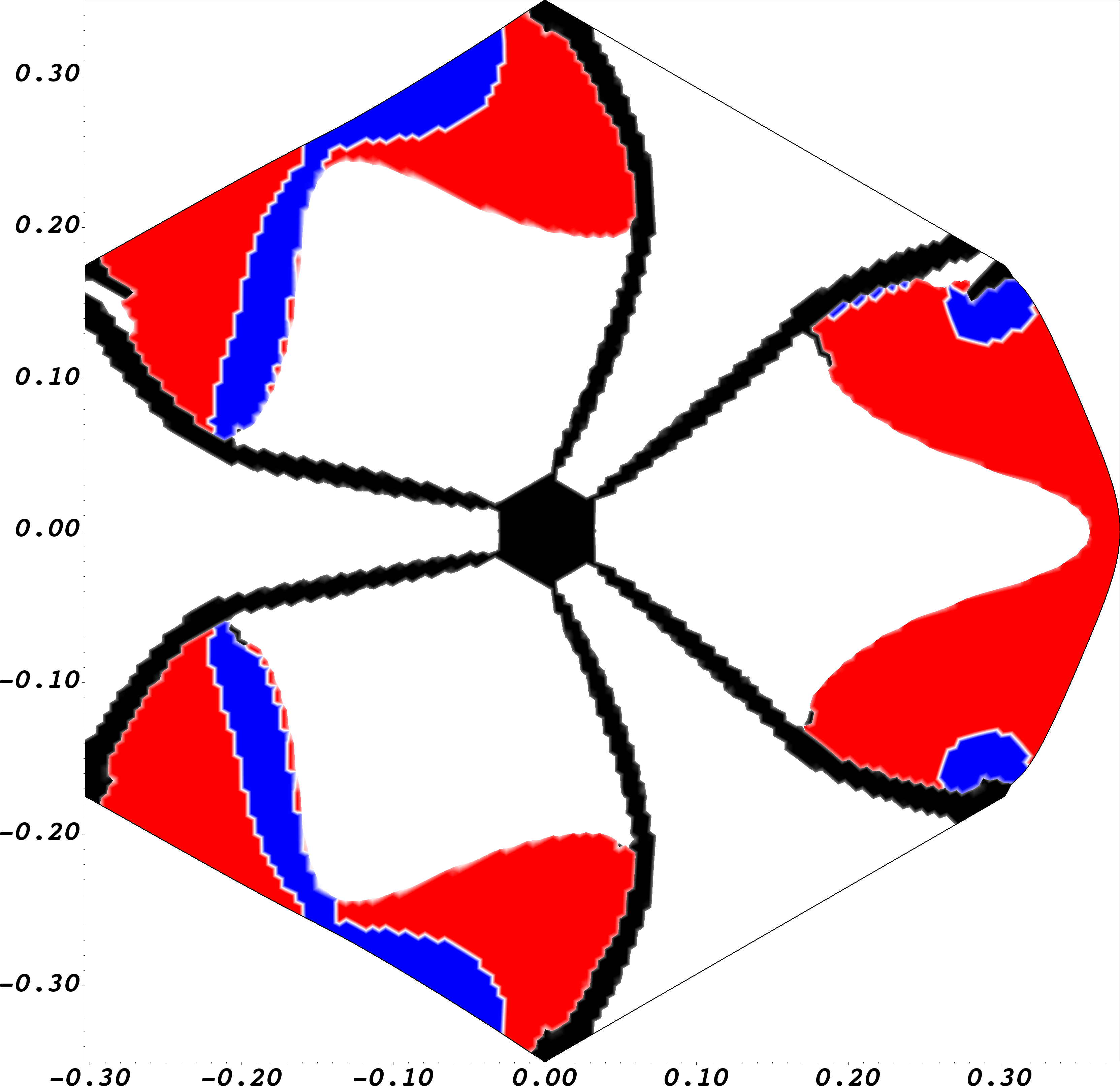}
    \includegraphics[width=0.3\textwidth]{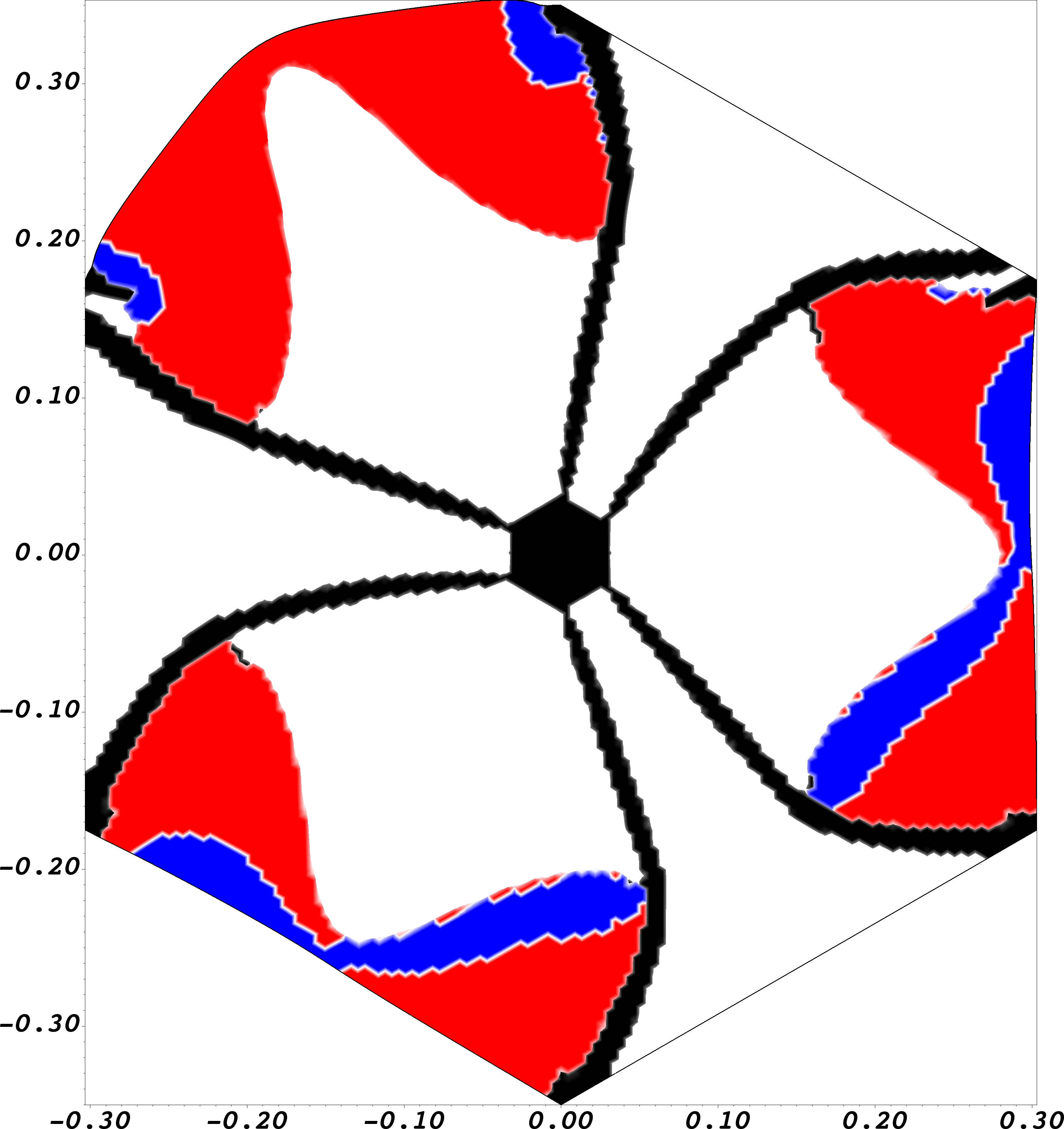}
    \includegraphics[width=0.3\textwidth]{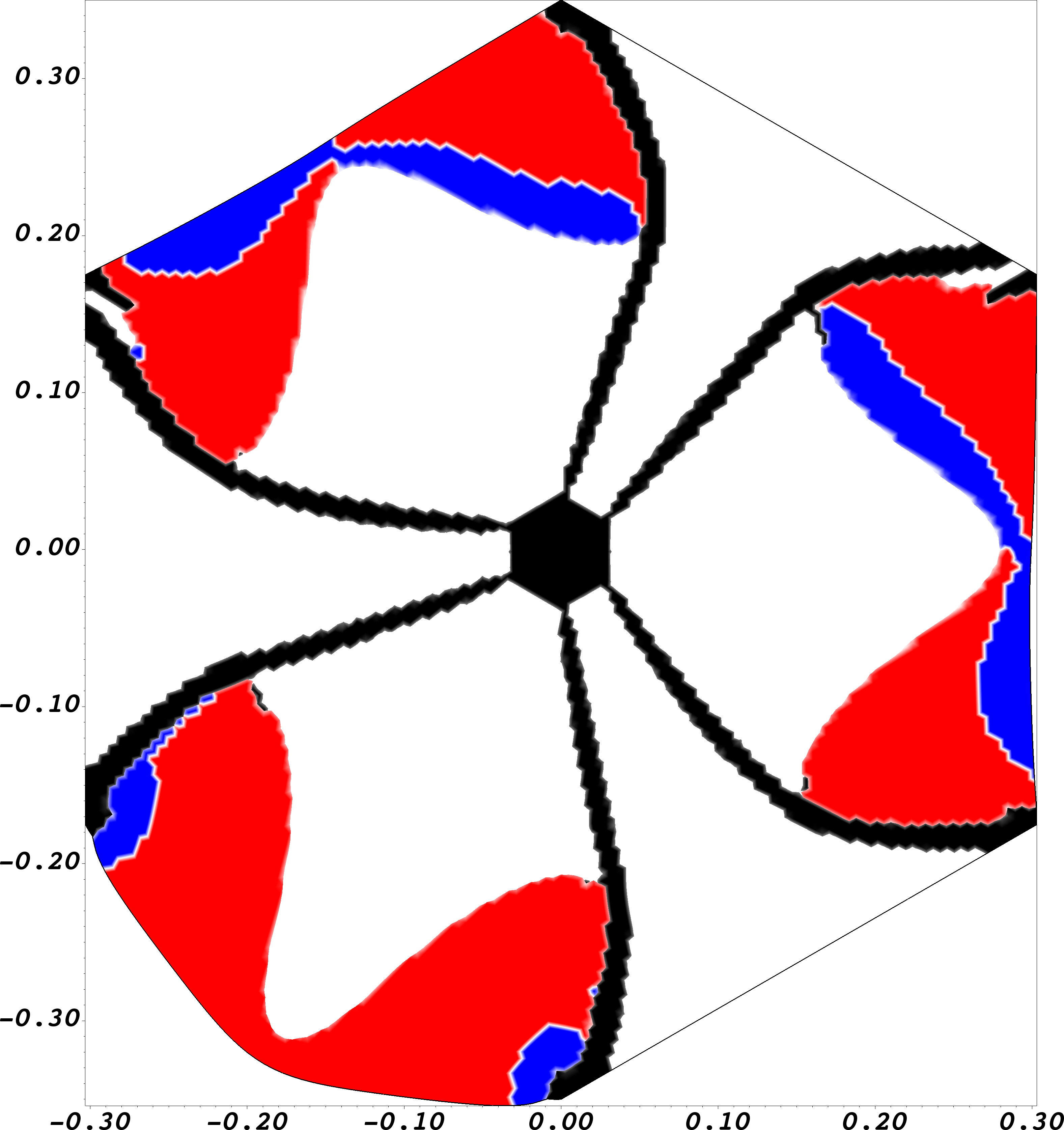}
    \caption{Optimal design with three prescribed displacement with $E_2/E_3 = 5$ and a strong penalty on the stiff material.}
    \label{fig:Hexagon-Ratio05-04}
\end{figure}

Finally, when using a stiff responsive material and weaker non-responsive materials ($E_2 = \num{1e-2}$ and $E_3 = \num{5e-2}$), all other parameters remaining the same,
we obtain designs consisting solely of the responsive materials.  
This is expected since in this situation, the non-responsive material is both expensive, less stiff, and incapable of activation.
\begin{figure}[h!]
    \centering
    \includegraphics[width=0.3\textwidth]{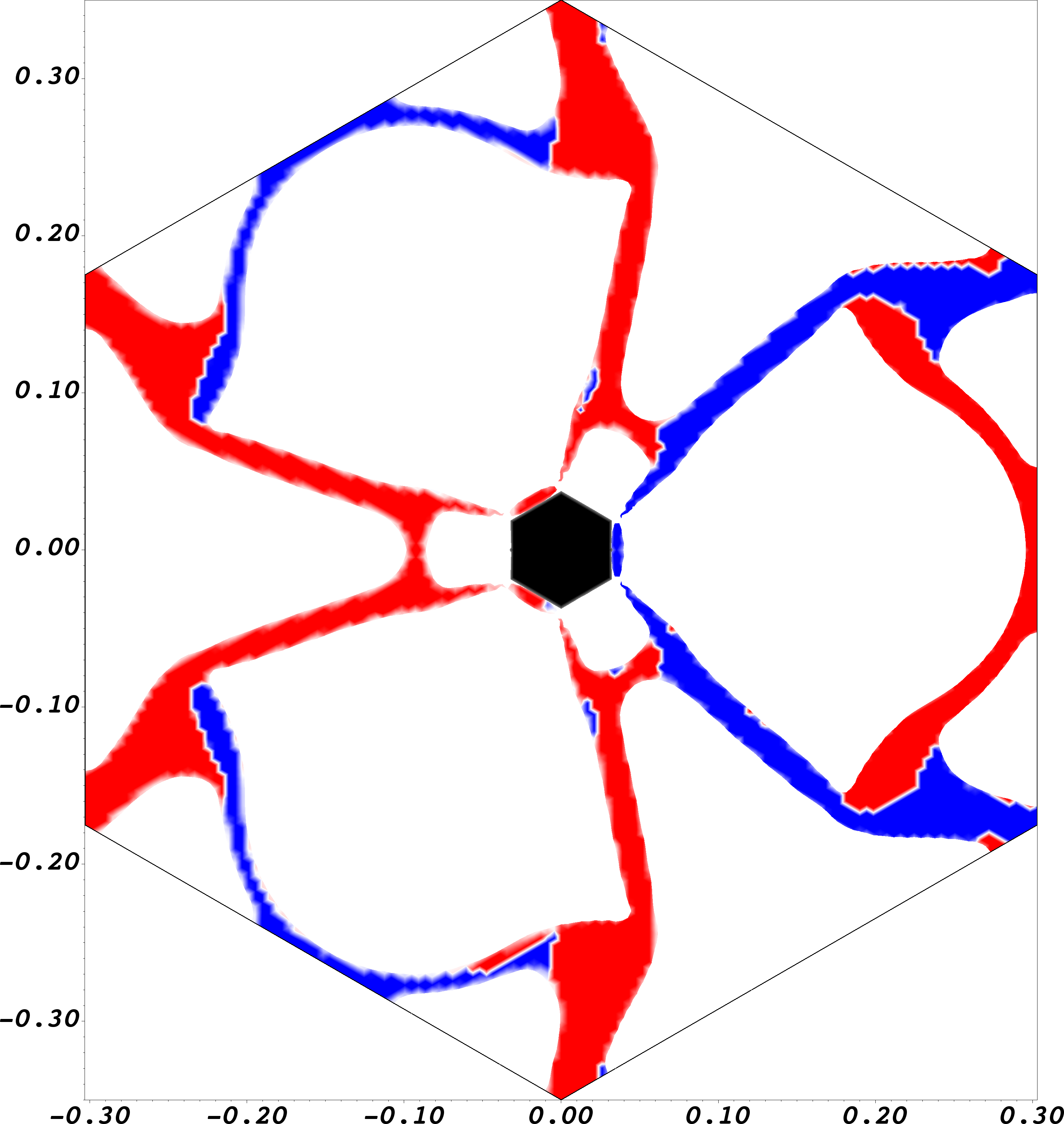}
    \includegraphics[width=0.3\textwidth]{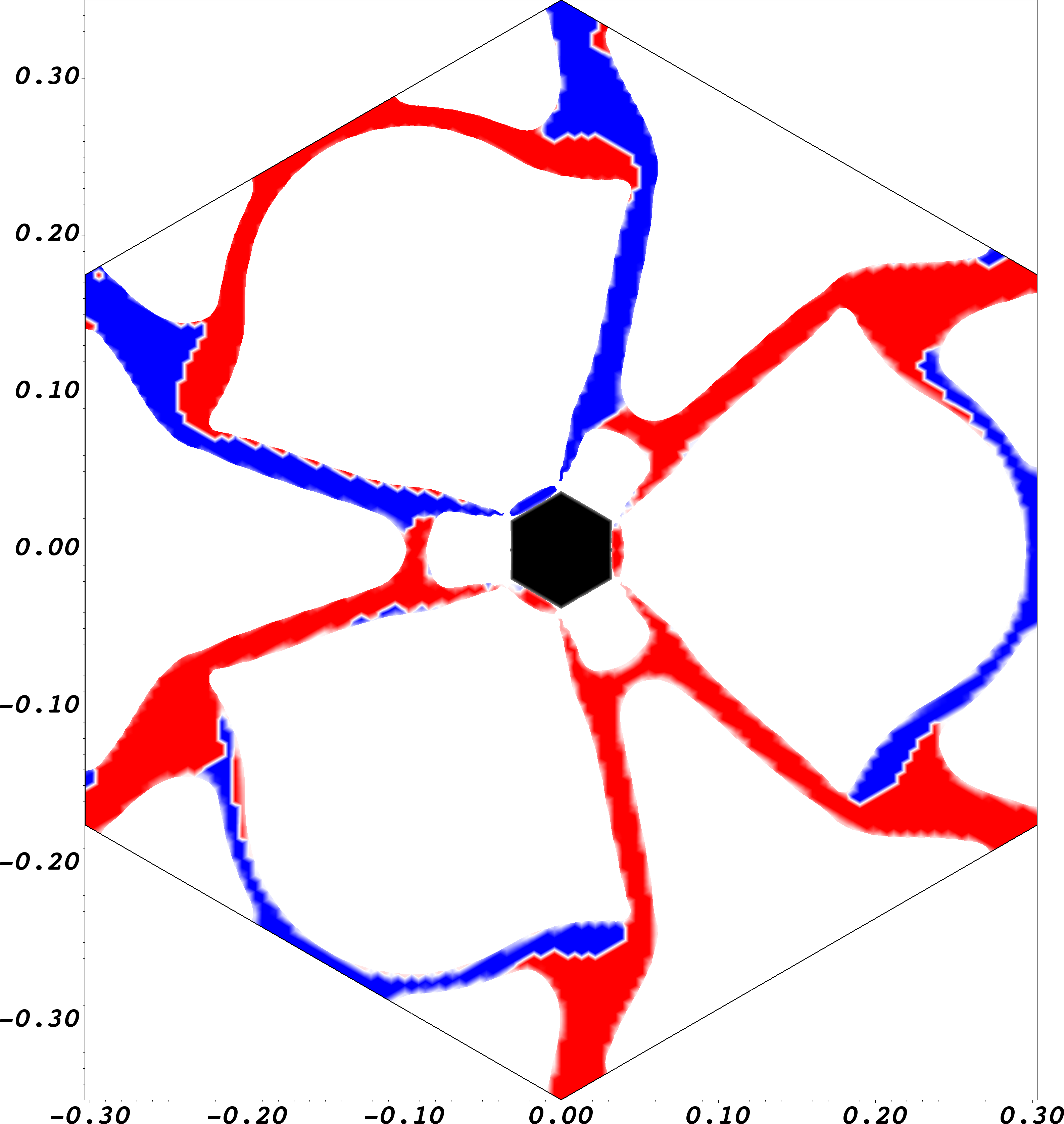}
    \includegraphics[width=0.3\textwidth]{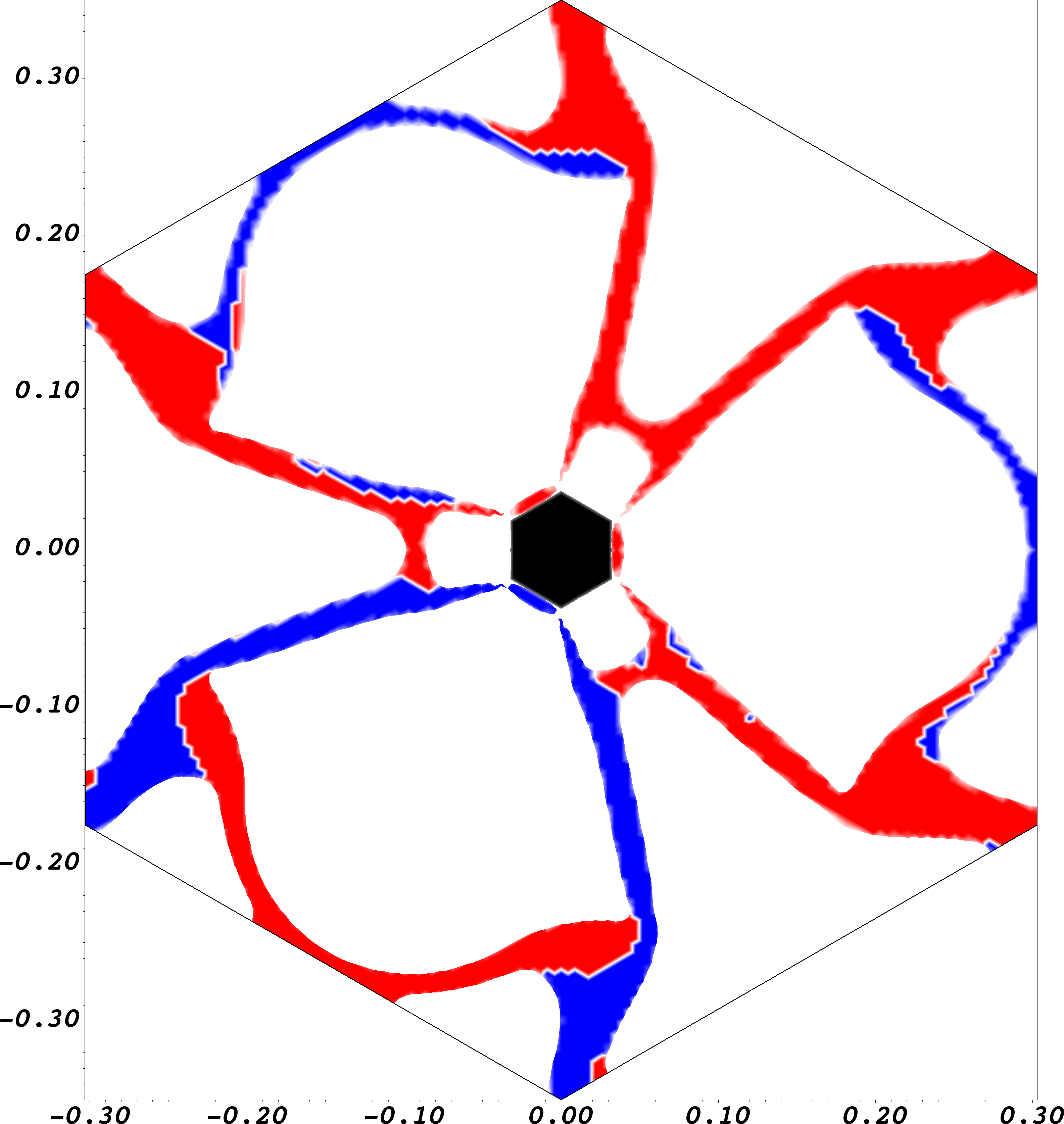}
    \includegraphics[width=0.3\textwidth]{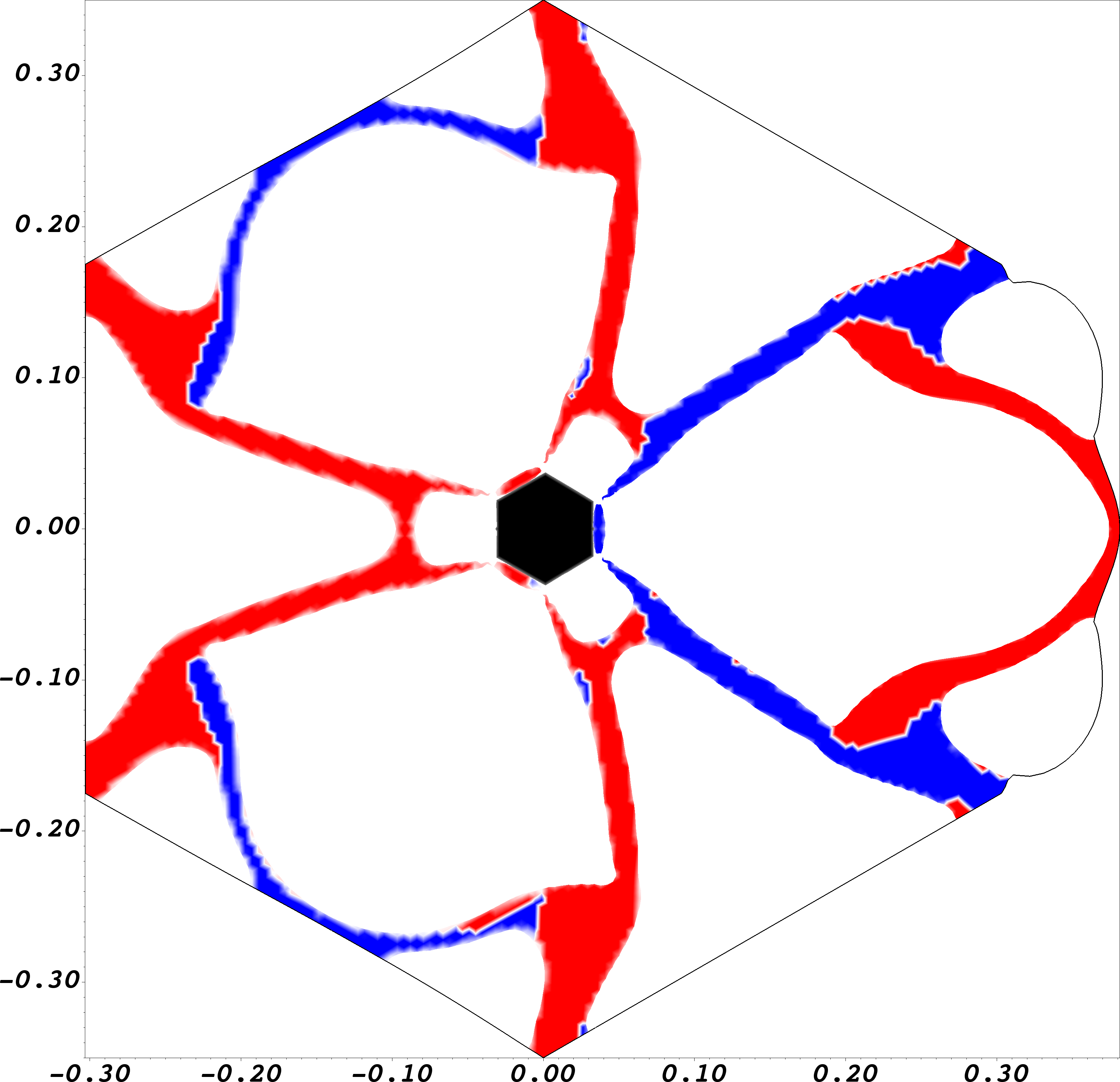}
    \includegraphics[width=0.3\textwidth]{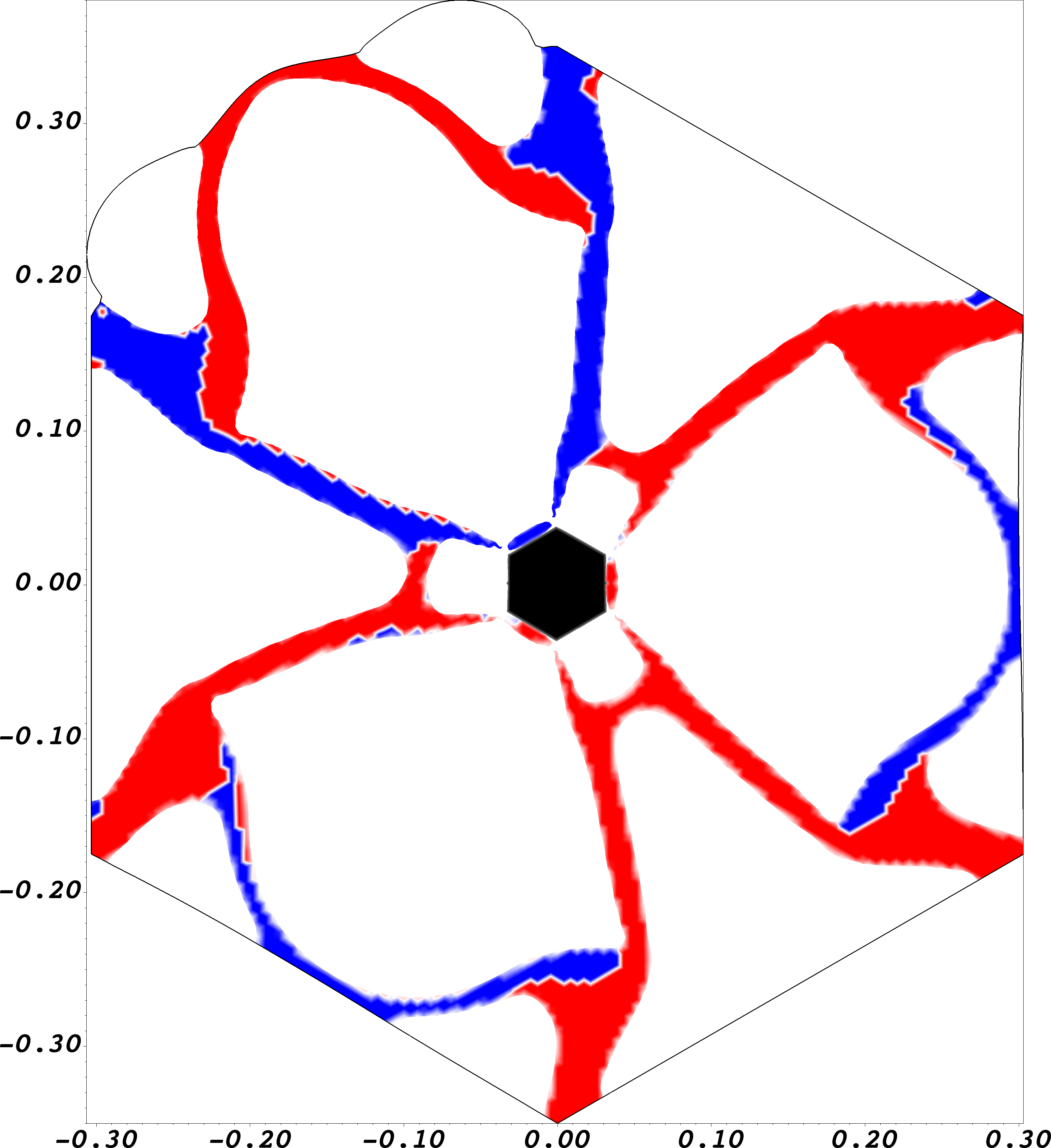}
    \includegraphics[width=0.3\textwidth]{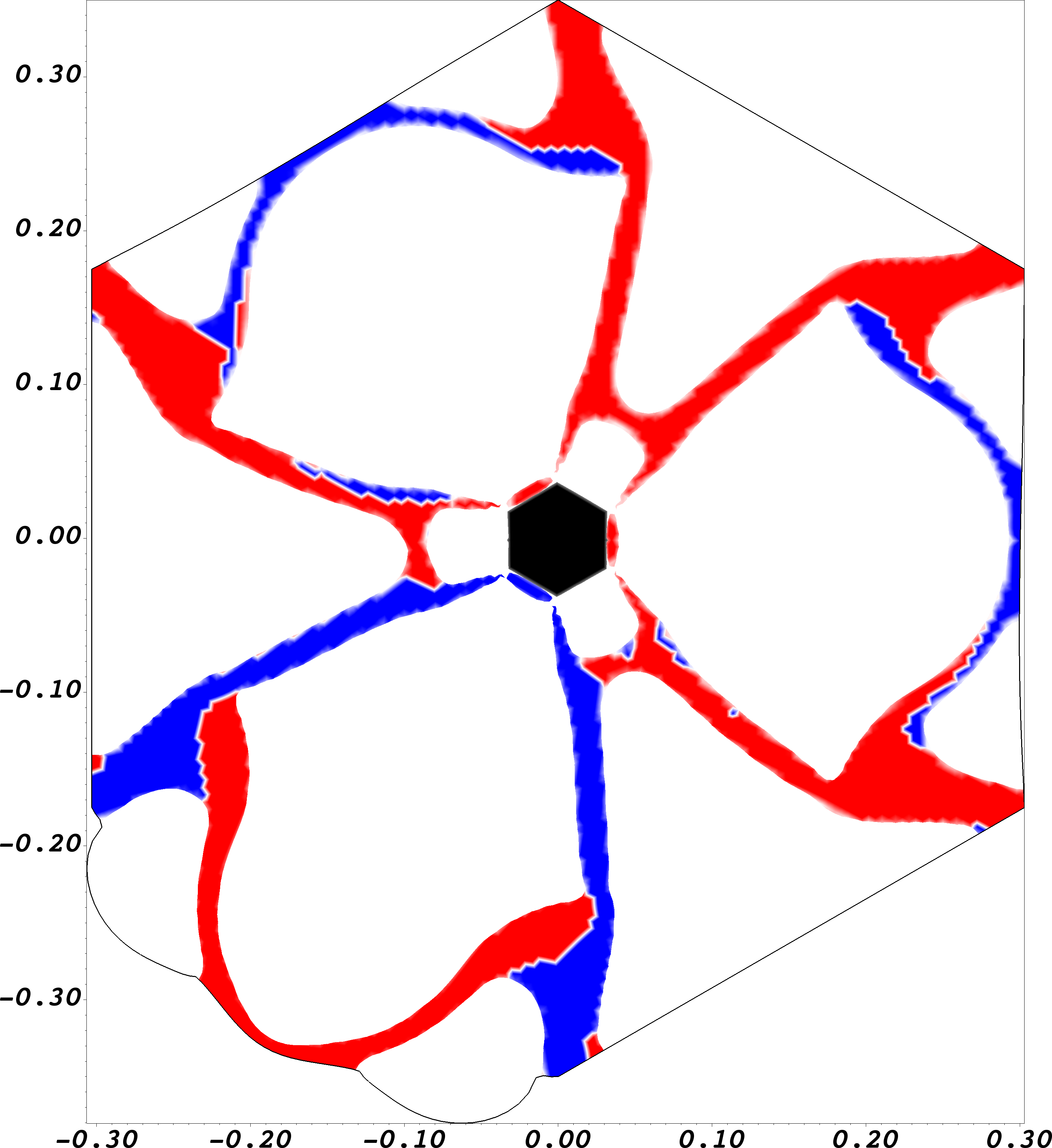}
    \caption{Optimal design with three prescribed displacement with $E_2/E_3 = 0.2$ and a strong penalty on the stiff material.}
    \label{fig:Hexagon-Ratio05-18}
\end{figure}

\section{Conclusion}\label{conclusion}
We have investigated the systematic design of responsive structures with a prescribed target displacement and perimeter constraint.
We proved existence of solutions to a phase-field regularization and their convergence to that of the ``sharp interface'' problem.
We proposed a numerical scheme based on an iterative gradient-based solver with respect to one set of design variable (the materials' density) where at each step a full minimization of the objective function with respect to the second set of design variables (the stimuli) is performed.
Our approach is illustrated by series of numerical examples demonstrating its ability to identify complex geometries and produced well-delineated ``black and white'' design, owing to the existence of ``classical'' solutions of the phase-field regularization.

While this work focused on a simple actuation mechanism in the form of an isotropic inelastic strain, accounting for to more complex stimuli and responsive materials including piezo-electrics, dielectric elastomers and liquid crystal elastomers or shape memory alloys should be a relatively simple extension of this work.
Whereas the objective function used here focus solely on kinematics of a structure, the framework could be extended to optimization of mechanical advantage as in~\cite{Akerson_2022,Sigmund-1997}.

A natural extension of the work presented in this article is to consider a stimulus derived from a physical process.
In this setting, the stimulus itself would become a state variable, derived from solving  a PDE governed by some additional design variable.
For instance, one could consider body or boundary heat flux as design variables, and the temperature field as a state variable.

\section*{Statements \& Declarations}
\paragraph{Funding} 
Support for this work was provided in part by the grant 
``Collaborative Research: Optimal Design of Responsive Materials and Structures'' from the U.S. National Science Foundation (DMS:2009303 at Louisiana State University and McMaster University and DMS:2009289 at Caltech).
BB also acknowledges the support of the Natural Sciences and Engineering Research Council of Canada (NSERC), RGPIN-2022-04536 and the Canada Research Chair program.

\paragraph{Competing Interests}
The authors have no relevant financial or non-financial interests to disclose.

\paragraph{Authors contribution}
All authors contributed to the study and manuscript preparation. All authors read and approved the final manuscript.

\paragraph{Code and data availability}
The code used in this article is available under an open source license at \url{https://doi.org/10.5281/zenodo.12746321}. Data files and computation results for Figure 1-8 are hosted by the Federated Research Data Repository (FRDR) of the Digital Research Alliance of Canada at \url{https://doi.org/10.20383/103.01009}.


\begin{thebibliography}{35}
\ifx \bisbn   \undefined \def \bisbn  #1{ISBN #1}\fi
\ifx \binits  \undefined \def \binits#1{#1}\fi
\ifx \bauthor  \undefined \def \bauthor#1{#1}\fi
\ifx \batitle  \undefined \def \batitle#1{#1}\fi
\ifx \bjtitle  \undefined \def \bjtitle#1{#1}\fi
\ifx \bvolume  \undefined \def \bvolume#1{\textbf{#1}}\fi
\ifx \byear  \undefined \def \byear#1{#1}\fi
\ifx \bissue  \undefined \def \bissue#1{#1}\fi
\ifx \bfpage  \undefined \def \bfpage#1{#1}\fi
\ifx \blpage  \undefined \def \blpage #1{#1}\fi
\ifx \burl  \undefined \def \burl#1{\textsf{#1}}\fi
\ifx \doiurl  \undefined \def \doiurl#1{\url{https://doi.org/#1}}\fi
\ifx \betal  \undefined \def \betal{\textit{et al.}}\fi
\ifx \binstitute  \undefined \def \binstitute#1{#1}\fi
\ifx \binstitutionaled  \undefined \def \binstitutionaled#1{#1}\fi
\ifx \bctitle  \undefined \def \bctitle#1{#1}\fi
\ifx \beditor  \undefined \def \beditor#1{#1}\fi
\ifx \bpublisher  \undefined \def \bpublisher#1{#1}\fi
\ifx \bbtitle  \undefined \def \bbtitle#1{#1}\fi
\ifx \bedition  \undefined \def \bedition#1{#1}\fi
\ifx \bseriesno  \undefined \def \bseriesno#1{#1}\fi
\ifx \blocation  \undefined \def \blocation#1{#1}\fi
\ifx \bsertitle  \undefined \def \bsertitle#1{#1}\fi
\ifx \bsnm \undefined \def \bsnm#1{#1}\fi
\ifx \bsuffix \undefined \def \bsuffix#1{#1}\fi
\ifx \bparticle \undefined \def \bparticle#1{#1}\fi
\ifx \barticle \undefined \def \barticle#1{#1}\fi
\bibcommenthead
\ifx \bconfdate \undefined \def \bconfdate #1{#1}\fi
\ifx \botherref \undefined \def \botherref #1{#1}\fi
\ifx \url \undefined \def \url#1{\textsf{#1}}\fi
\ifx \bchapter \undefined \def \bchapter#1{#1}\fi
\ifx \bbook \undefined \def \bbook#1{#1}\fi
\ifx \bcomment \undefined \def \bcomment#1{#1}\fi
\ifx \oauthor \undefined \def \oauthor#1{#1}\fi
\ifx \citeauthoryear \undefined \def \citeauthoryear#1{#1}\fi
\ifx \endbibitem  \undefined \def \endbibitem {}\fi
\ifx \bconflocation  \undefined \def \bconflocation#1{#1}\fi
\ifx \arxivurl  \undefined \def \arxivurl#1{\textsf{#1}}\fi
\csname PreBibitemsHook\endcsname

\bibitem[\protect\citeauthoryear{Schaedler
  et~al.}{2011}]{Schaedler-Jacobsen-EtAl-2011b}
\begin{barticle}
\bauthor{\bsnm{Schaedler}, \binits{T.A.}},
\bauthor{\bsnm{Jacobsen}, \binits{A.J.}},
\bauthor{\bsnm{Torrents}, \binits{A.}},
\bauthor{\bsnm{Sorensen}, \binits{A.E.}},
\bauthor{\bsnm{Lian}, \binits{J.}},
\bauthor{\bsnm{Greer}, \binits{J.R.}},
\bauthor{\bsnm{Valdevit}, \binits{L.}},
\bauthor{\bsnm{Carter}, \binits{W.B.}}:
\batitle{Ultralight metallic microlattices}.
\bjtitle{Science}
\bvolume{334}(\bissue{6058}),
\bfpage{962}--\blpage{965}
(\byear{2011})
\doiurl{10.1126/science.1211649}
\end{barticle}
\endbibitem

\bibitem[\protect\citeauthoryear{Xu et~al.}{2019}]{Xu-Zhang-EtAl-2019a}
\begin{botherref}
\oauthor{\bsnm{Xu}, \binits{T.}},
\oauthor{\bsnm{Zhang}, \binits{J.}},
\oauthor{\bsnm{Salehizadeh}, \binits{M.}},
\oauthor{\bsnm{Onaizah}, \binits{O.}},
\oauthor{\bsnm{Diller}, \binits{E.}}:
Millimeter-scale flexible robots with programmable three-dimensional
  magnetization and motions.
Science Robotics
\textbf{4}(29)
(2019)
\doiurl{10.1126/scirobotics.aav4494}
\end{botherref}
\endbibitem

\bibitem[\protect\citeauthoryear{Sigmund}{1997}]{Sigmund-1997}
\begin{barticle}
\bauthor{\bsnm{Sigmund}, \binits{O.}}:
\batitle{On the design of compliant mechanisms using topology optimization}.
\bjtitle{Mechanics of Structures and Machines}
\bvolume{25}(\bissue{4}),
\bfpage{495}--\blpage{526}
(\byear{1997})
\doiurl{10.1080/08905459708945415}
\end{barticle}
\endbibitem

\bibitem[\protect\citeauthoryear{Larsen
  et~al.}{1997}]{Larsen-Sigmund-EtAl-1997}
\begin{barticle}
\bauthor{\bsnm{Larsen}, \binits{U.D.}},
\bauthor{\bsnm{Sigmund}, \binits{O.}},
\bauthor{\bsnm{Bouwstra}, \binits{S.}}:
\batitle{Design and fabrication of micromecanisms and structures with negative
  {P}oisson's ratio}.
\bjtitle{Journal of Microelectromechanical Systems}
\bvolume{6}(\bissue{2}),
\bfpage{99}--\blpage{106}
(\byear{1997})
\doiurl{10.1109/84.585787}
\end{barticle}
\endbibitem

\bibitem[\protect\citeauthoryear{Jonsmann
  et~al.}{1999}]{Jonsmann-Sigmund-EtAl-1999}
\begin{barticle}
\bauthor{\bsnm{Jonsmann}, \binits{J.}},
\bauthor{\bsnm{Sigmund}, \binits{O.}},
\bauthor{\bsnm{Bouwstra}, \binits{S.}}:
\batitle{Compliant thermal microactuators}.
\bjtitle{Sensors and Actuators}
\bvolume{7},
\bfpage{463}--\blpage{469}
(\byear{1999})
\doiurl{10.1016/S0924-4247(99)00011-4}
\end{barticle}
\endbibitem

\bibitem[\protect\citeauthoryear{Bends{\o}e and
  Sigmund}{2003}]{BendsoeSigmund2003}
\begin{bbook}
\bauthor{\bsnm{Bends{\o}e}, \binits{M.}},
\bauthor{\bsnm{Sigmund}, \binits{O.}}:
\bbtitle{Topology Optimization: Theory, Methods and Applications},
\bedition{2nd} edn.
\bpublisher{Springer},
\blocation{New York}
(\byear{2003})
\end{bbook}
\endbibitem

\bibitem[\protect\citeauthoryear{Allaire}{2002}]{Allaire2002}
\begin{bbook}
\bauthor{\bsnm{Allaire}, \binits{G.}}:
\bbtitle{Shape Optimization by the Homogenization Method},
p. \bfpage{456}.
\bpublisher{Springer},
\blocation{New York}
(\byear{2002})
\end{bbook}
\endbibitem

\bibitem[\protect\citeauthoryear{Andrej}{2000}]{Cherkaev2000}
\begin{bbook}
\bauthor{\bsnm{Andrej}, \binits{C.}}:
\bbtitle{Variational Methods for Structural Optimization}.
\bsertitle{Applied Mathematical Sciences},
vol. \bseriesno{140}.
\bpublisher{Springer},
\blocation{Berlin}
(\byear{2000}).
\doiurl{10.1007/978-1-4612-1188-4}
\end{bbook}
\endbibitem

\bibitem[\protect\citeauthoryear{Allaire et~al.}{1997}]{ABFJ1997}
\begin{barticle}
\bauthor{\bsnm{Allaire}, \binits{G.}},
\bauthor{\bsnm{Bonnetier}, \binits{E.}},
\bauthor{\bsnm{Francfort}, \binits{G.}},
\bauthor{},
\bauthor{\bsnm{Jouve}, \binits{F.}}:
\batitle{Shape optimization by the homogenization method.}
\bjtitle{Numerische Mathematik}
\bvolume{76}(\bissue{1}),
\bfpage{27}--\blpage{68}
(\byear{1997})
\doiurl{10.1007/s002110050253}
\end{barticle}
\endbibitem

\bibitem[\protect\citeauthoryear{Bends{\o}e and
  Sigmund}{1999}]{Bendsoe-Sigmund-1999}
\begin{barticle}
\bauthor{\bsnm{Bends{\o}e}, \binits{M.P.}},
\bauthor{\bsnm{Sigmund}, \binits{O.}}:
\batitle{Material interpolation schemes in topology optimization}.
\bjtitle{Archive of Applied Mechanics}
\bvolume{69},
\bfpage{635}--\blpage{654}
(\byear{1999})
\doiurl{10.1007/s004190050248}
\end{barticle}
\endbibitem

\bibitem[\protect\citeauthoryear{Bourdin}{2001}]{Bourdin-2001}
\begin{barticle}
\bauthor{\bsnm{Bourdin}, \binits{B.}}:
\batitle{Filters in topology optimization}.
\bjtitle{International Journal for Numerical Methods in Engineering}
\bvolume{50},
\bfpage{2143}--\blpage{2158}
(\byear{2001})
\doiurl{10.1002/nme.116}
\end{barticle}
\endbibitem

\bibitem[\protect\citeauthoryear{Allaire et~al.}{2004}]{Allaire_2004}
\begin{barticle}
\bauthor{\bsnm{Allaire}, \binits{G.}},
\bauthor{\bsnm{Jouve}, \binits{F.}},
\bauthor{\bsnm{Toader}, \binits{A.-M.}}:
\batitle{Structural optimization using sensitivity analysis and a level-set
  method}.
\bjtitle{Journal of Computational Physics}
\bvolume{194}(\bissue{1}),
\bfpage{363}--\blpage{393}
(\byear{2004})
\doiurl{10.1016/j.jcp.2003.09.032}
\end{barticle}
\endbibitem

\bibitem[\protect\citeauthoryear{Allaire et~al.}{2013}]{Allaire_2013}
\begin{barticle}
\bauthor{\bsnm{Allaire}, \binits{G.}},
\bauthor{\bsnm{Dapogny}, \binits{C.}},
\bauthor{\bsnm{Frey}, \binits{P.}}:
\batitle{A mesh evolution algorithm based on the level set method for geometry
  and topology optimization}.
\bjtitle{Structural and Multidisciplinary Optimization}
\bvolume{48}(\bissue{4}),
\bfpage{711}--\blpage{715}
(\byear{2013})
\doiurl{10.1007/s00158-013-0929-2}
\end{barticle}
\endbibitem

\bibitem[\protect\citeauthoryear{Ambrosio and
  Buttazzo}{1993}]{AmbrosioButtazzo1993}
\begin{barticle}
\bauthor{\bsnm{Ambrosio}, \binits{L.}},
\bauthor{\bsnm{Buttazzo}, \binits{G.}}:
\batitle{An optimal design problem with perimeter penalization.}
\bjtitle{Calculus of Variations and Partial Differential Equations}
\bvolume{1}(\bissue{1}),
\bfpage{55}--\blpage{59}
(\byear{1993})
\doiurl{10.1016/j.jcp.2003.09.032}
\end{barticle}
\endbibitem

\bibitem[\protect\citeauthoryear{Haber et~al.}{1996}]{Haber-EtAl-1996}
\begin{barticle}
\bauthor{\bsnm{Haber}, \binits{R.B.}},
\bauthor{\bsnm{Jog}, \binits{C.S.}},
\bauthor{\bsnm{Bends{\o}e}, \binits{M.P.}}:
\batitle{A new approach to variable-topology shape design using a constraint on
  the perimeter}.
\bjtitle{Structural and Multidisciplinary Optimization}
\bvolume{11},
\bfpage{1}--\blpage{12}
(\byear{1996})
\doiurl{10.1007/BF01279647}
\end{barticle}
\endbibitem

\bibitem[\protect\citeauthoryear{Petersson}{1999}]{Petersson-1999b}
\begin{barticle}
\bauthor{\bsnm{Petersson}, \binits{J.}}:
\batitle{Some convergence results in perimeter-controlled topology
  optimization}.
\bjtitle{Computer Methods in Applied Mechanics and Engineering}
\bvolume{171},
\bfpage{123}--\blpage{140}
(\byear{1999})
\doiurl{10.1016/S0045-7825(98)00248-5}
\end{barticle}
\endbibitem

\bibitem[\protect\citeauthoryear{Bourdin and
  Chambolle}{2003}]{BourdinChambolle2003}
\begin{barticle}
\bauthor{\bsnm{Bourdin}, \binits{B.}},
\bauthor{\bsnm{Chambolle}, \binits{A.}}:
\batitle{Design-dependent loads in topology optimization}.
\bjtitle{{ESAIM}: Control, Optimisation and Calculus of Variations}
\bvolume{9},
\bfpage{19}--\blpage{48}
(\byear{2003})
\doiurl{10.1051/cocv:2002070}
\end{barticle}
\endbibitem

\bibitem[\protect\citeauthoryear{Bourdin and
  Chambolle}{2006}]{BourdinChambolle2006}
\begin{bchapter}
\bauthor{\bsnm{Bourdin}, \binits{B.}},
\bauthor{\bsnm{Chambolle}, \binits{A.}}:
\bctitle{The phase-field method in optimal design}.
In: \beditor{\bsnm{M.P.~Bends{\o}e}, \binits{O.S.}},
\beditor{\bsnm{Olhoff}, \binits{N.}} (eds.)
\bbtitle{IUTAM Symposium on Topological Design Optimization of Structures,
  Machines and Materials}.
\bsertitle{Solid Mechanics and its Applications},
pp. \bfpage{207}--\blpage{216}.
\bpublisher{Springer},
\blocation{Dordrecht}
(\byear{2006}).
\doiurl{10.1007/1-4020-4752-5}
\end{bchapter}
\endbibitem

\bibitem[\protect\citeauthoryear{Tran and Bourdin}{2022}]{Tran-Bourdin-2022c}
\begin{botherref}
\oauthor{\bsnm{Tran}, \binits{N.V.}},
\oauthor{\bsnm{Bourdin}, \binits{B.}}:
Minimum compliance with obstacle constraints: an active set approach.
Structural and Multidisciplinary Optimization
\textbf{65}(4)
(2022)
\doiurl{10.1007/s00158-022-03199-9}
\end{botherref}
\endbibitem

\bibitem[\protect\citeauthoryear{Ambrosio et~al.}{2000}]{Ambrosio-EtAl-2000}
\begin{bbook}
\bauthor{\bsnm{Ambrosio}, \binits{L.}},
\bauthor{\bsnm{Fusco}, \binits{N.}},
\bauthor{\bsnm{Pallara}, \binits{D.}}:
\bbtitle{Functions of Bounded Variation and Free Discontinuity Problems}.
\bpublisher{Oxford University Press},
\blocation{New York}
(\byear{2000}).
\doiurl{10.1093/oso/9780198502456.001.0001}
\end{bbook}
\endbibitem

\bibitem[\protect\citeauthoryear{Wang and Zhou}{2004}]{Wang-Zhou-2004a}
\begin{barticle}
\bauthor{\bsnm{Wang}, \binits{M.Y.}},
\bauthor{\bsnm{Zhou}, \binits{S.}}:
\batitle{Synthesis of shape and topology of multi-material structures with a
  phase-field method}.
\bjtitle{Journal of Computer-Aided Materials Design}
\bvolume{11}(\bissue{2-3}),
\bfpage{117}--\blpage{138}
(\byear{2004})
\doiurl{10.1007/s10820-005-3169-y}
\end{barticle}
\endbibitem

\bibitem[\protect\citeauthoryear{Zhou and Wang}{2007}]{Zhou-Wang-2007}
\begin{barticle}
\bauthor{\bsnm{Zhou}, \binits{S.W.}},
\bauthor{\bsnm{Wang}, \binits{M.Y.}}:
\batitle{Multimaterial structural topology optimization with a generalized
  {C}ahn-{H}illiard model of multiphase transition}.
\bjtitle{Structural and Multidisciplinary Optimization}
\bvolume{33}(\bissue{2}),
\bfpage{89}--\blpage{111}
(\byear{2007})
\doiurl{10.1007/s00158-006-0035-9}
\end{barticle}
\endbibitem

\bibitem[\protect\citeauthoryear{Garcke et~al.}{2021}]{Garcke-Huttl-EtAl-2021a}
\begin{barticle}
\bauthor{\bsnm{Garcke}, \binits{H.}},
\bauthor{\bsnm{H{\"u}ttl}, \binits{P.}},
\bauthor{\bsnm{Knopf}, \binits{P.}}:
\batitle{Shape and topology optimization involving the eigenvalues of an
  elastic structure: A multi-phase-field approach}.
\bjtitle{Advances in Nonlinear Analysis}
\bvolume{11}(\bissue{1}),
\bfpage{159}--\blpage{197}
(\byear{2021})
\doiurl{10.1515/anona-2020-0183}
\end{barticle}
\endbibitem

\bibitem[\protect\citeauthoryear{Garcke et~al.}{2024}]{Garcke-Huttl-EtAl-2024a}
\begin{botherref}
\oauthor{\bsnm{Garcke}, \binits{H.}},
\oauthor{\bsnm{H{\"u}ttl}, \binits{P.}},
\oauthor{\bsnm{Kahle}, \binits{C.}},
\oauthor{\bsnm{Knopf}, \binits{P.}}:
Sharp-interface limit of a multi-phase spectral shape optimization problem for
  elastic structures.
Applied Mathematics \& Optimization
\textbf{89}(1)
(2024)
\doiurl{10.1007/s00245-023-10093-3}
\end{botherref}
\endbibitem

\bibitem[\protect\citeauthoryear{Alberti}{2000}]{Alberti-2000}
\begin{bchapter}
\bauthor{\bsnm{Alberti}, \binits{G.}}:
\bctitle{Variational models for phase transitions, an approach via
  {$\Gamma$}--convergence}.
In: \beditor{\bsnm{Buttazzo}, \binits{G.}},
\beditor{\bsnm{Marino}, \binits{M.}},
\beditor{\bsnm{Murthy}, \binits{M.K.V.}} (eds.)
\bbtitle{Calculus of Variations and Partial Differential Equations},
pp. \bfpage{95}--\blpage{114}.
\bpublisher{Springer},
\blocation{Berlin}
(\byear{2000}).
\doiurl{10.1007/978-3-642-57186-2_3}
\end{bchapter}
\endbibitem

\bibitem[\protect\citeauthoryear{Modica and
  Mortola}{1977}]{Modica-Mortola-1977b}
\begin{barticle}
\bauthor{\bsnm{Modica}, \binits{L.}},
\bauthor{\bsnm{Mortola}, \binits{S.}}:
\batitle{Un esempio di {$\Gamma$}--convergenza}.
\bjtitle{Bollettino dell'Unione Matematica Italiana B (5)}
\bvolume{14}(\bissue{1}),
\bfpage{285}--\blpage{299}
(\byear{1977})
\end{barticle}
\endbibitem

\bibitem[\protect\citeauthoryear{Modica}{1987}]{Modica1987a}
\begin{barticle}
\bauthor{\bsnm{Modica}, \binits{L.}}:
\batitle{The gradient theory of phase transitions and the minimal interface
  criterion}.
\bjtitle{Archive for Rational Mechanics and Analysis}
\bvolume{98}(\bissue{2}),
\bfpage{123}--\blpage{142}
(\byear{1987})
\doiurl{10.1007/BF00251230}
\end{barticle}
\endbibitem

\bibitem[\protect\citeauthoryear{Baldo}{1990}]{Baldo-1990a}
\begin{barticle}
\bauthor{\bsnm{Baldo}, \binits{S.}}:
\batitle{Minimal interface criterion for phase transitions in mixtures of
  {C}ahn-{H}illiard fluids}.
\bjtitle{Annales de l'{I}nstitut {H}enri {P}oincar{\'e} Analyse non lin\'eaire}
\bvolume{7}(\bissue{2}),
\bfpage{67}--\blpage{90}
(\byear{1990})
\doiurl{AIHPC_1990__7_2_67_0}
\end{barticle}
\endbibitem

\bibitem[\protect\citeauthoryear{Ciarlet}{2013}]{Ciarlet-2013a}
\begin{bbook}
\bauthor{\bsnm{Ciarlet}, \binits{P.G.}}:
\bbtitle{Linear and Nonlinear Functional Analysis with Applications}.
\bpublisher{SIAM},
\blocation{Philadelphia}
(\byear{2013}).
\doiurl{10.1137/1.9781611972597.fm}
\end{bbook}
\endbibitem

\bibitem[\protect\citeauthoryear{De~los {R}eyes}{2015}]{DelosReyes2015}
\begin{bbook}
\bauthor{\bsnm{{R}eyes}, \binits{J.C.}}:
\bbtitle{Numerical {PDE}-Constrained Optimization}.
\bsertitle{SpringerBriefs in Optimization}.
\bpublisher{Springer},
\blocation{New York}
(\byear{2015}).
\doiurl{10.1007/978-3-319-13395-9}
\end{bbook}
\endbibitem

\bibitem[\protect\citeauthoryear{Ham et~al.}{2023}]{FiredrakeUserManual}
\begin{botherref}
\oauthor{\bsnm{Ham}, \binits{D.A.}},
\oauthor{\bsnm{Kelly}, \binits{P.H.J.}},
\oauthor{\bsnm{Mitchell}, \binits{L.}},
\oauthor{\bsnm{Cotter}, \binits{C.J.}},
\oauthor{\bsnm{Kirby}, \binits{R.C.}},
\oauthor{\bsnm{Sagiyama}, \binits{K.}},
\oauthor{\bsnm{Bouziani}, \binits{N.}},
\oauthor{\bsnm{Vorderwuelbecke}, \binits{S.}},
\oauthor{\bsnm{Gregory}, \binits{T.J.}},
\oauthor{\bsnm{Betteridge}, \binits{J.}},
\oauthor{\bsnm{Shapero}, \binits{D.R.}},
\oauthor{\bsnm{Nixon-Hill}, \binits{R.W.}},
\oauthor{\bsnm{Ward}, \binits{C.J.}},
\oauthor{\bsnm{Farrell}, \binits{P.E.}},
\oauthor{\bsnm{Brubeck}, \binits{P.D.}},
\oauthor{\bsnm{Marsden}, \binits{I.}},
\oauthor{\bsnm{Gibson}, \binits{T.H.}},
\oauthor{\bsnm{Homolya}, \binits{M.}},
\oauthor{\bsnm{Sun}, \binits{T.}},
\oauthor{\bsnm{McRae}, \binits{A.T.T.}},
\oauthor{\bsnm{Luporini}, \binits{F.}},
\oauthor{\bsnm{Gregory}, \binits{A.}},
\oauthor{\bsnm{Lange}, \binits{M.}},
\oauthor{\bsnm{Funke}, \binits{S.W.}},
\oauthor{\bsnm{Rathgeber}, \binits{F.}},
\oauthor{\bsnm{Bercea}, \binits{G.-T.}},
\oauthor{\bsnm{Markall}, \binits{G.R.}}:
Firedrake user manual.
Technical report,
Imperial College London
(2023).
\doiurl{10.25561/104839}
\end{botherref}
\endbibitem

\bibitem[\protect\citeauthoryear{Balay et~al.}{1997}]{petsc-efficient}
\begin{bchapter}
\bauthor{\bsnm{Balay}, \binits{S.}},
\bauthor{\bsnm{Gropp}, \binits{W.D.}},
\bauthor{\bsnm{McInnes}, \binits{L.C.}},
\bauthor{\bsnm{Smith}, \binits{B.F.}}:
\bctitle{Efficient management of parallelism in object oriented numerical
  software libraries}.
In: \beditor{\bsnm{Arge}, \binits{E.}},
\beditor{\bsnm{Bruaset}, \binits{A.M.}},
\beditor{\bsnm{Langtangen}, \binits{H.P.}} (eds.)
\bbtitle{Modern Software Tools in Scientific Computing},
pp. \bfpage{163}--\blpage{202}.
\bpublisher{Birkh{\"{a}}user Press},
\blocation{Boston}
(\byear{1997})
\end{bchapter}
\endbibitem

\bibitem[\protect\citeauthoryear{Balay et~al.}{2024a}]{petsc-user-ref}
\begin{botherref}
\oauthor{\bsnm{Balay}, \binits{S.}},
\oauthor{\bsnm{Abhyankar}, \binits{S.}},
\oauthor{\bsnm{Adams}, \binits{M.F.}},
\oauthor{\bsnm{Benson}, \binits{S.}},
\oauthor{\bsnm{Brown}, \binits{J.}},
\oauthor{\bsnm{Brune}, \binits{P.}},
\oauthor{\bsnm{Buschelman}, \binits{K.}},
\oauthor{\bsnm{Constantinescu}, \binits{E.}},
\oauthor{\bsnm{Dalcin}, \binits{L.}},
\oauthor{\bsnm{Dener}, \binits{A.}},
\oauthor{\bsnm{Eijkhout}, \binits{V.}},
\oauthor{\bsnm{Faibussowitsch}, \binits{J.}},
\oauthor{\bsnm{Gropp}, \binits{W.D.}},
\oauthor{\bsnm{Hapla}, \binits{V.}},
\oauthor{\bsnm{Isaac}, \binits{T.}},
\oauthor{\bsnm{Jolivet}, \binits{P.}},
\oauthor{\bsnm{Karpeev}, \binits{D.}},
\oauthor{\bsnm{Kaushik}, \binits{D.}},
\oauthor{\bsnm{Knepley}, \binits{M.G.}},
\oauthor{\bsnm{Kong}, \binits{F.}},
\oauthor{\bsnm{Kruger}, \binits{S.}},
\oauthor{\bsnm{May}, \binits{D.A.}},
\oauthor{\bsnm{McInnes}, \binits{L.C.}},
\oauthor{\bsnm{Mills}, \binits{R.T.}},
\oauthor{\bsnm{Mitchell}, \binits{L.}},
\oauthor{\bsnm{Munson}, \binits{T.}},
\oauthor{\bsnm{Roman}, \binits{J.E.}},
\oauthor{\bsnm{Rupp}, \binits{K.}},
\oauthor{\bsnm{Sanan}, \binits{P.}},
\oauthor{\bsnm{Sarich}, \binits{J.}},
\oauthor{\bsnm{Smith}, \binits{B.F.}},
\oauthor{\bsnm{Zampini}, \binits{S.}},
\oauthor{\bsnm{Zhang}, \binits{H.}},
\oauthor{\bsnm{Zhang}, \binits{H.}},
\oauthor{\bsnm{Zhang}, \binits{J.}}:
{PETSc/TAO} users manual.
Technical Report ANL-21/39 - Revision 3.21,
Argonne National Laboratory
(2024).
\doiurl{10.2172/2205494}
\end{botherref}
\endbibitem

\bibitem[\protect\citeauthoryear{Balay et~al.}{2024b}]{petsc-web-page}
\begin{botherref}
\oauthor{\bsnm{Balay}, \binits{S.}},
\oauthor{\bsnm{Abhyankar}, \binits{S.}},
\oauthor{\bsnm{Adams}, \binits{M.F.}},
\oauthor{\bsnm{Benson}, \binits{S.}},
\oauthor{\bsnm{Brown}, \binits{J.}},
\oauthor{\bsnm{Brune}, \binits{P.}},
\oauthor{\bsnm{Buschelman}, \binits{K.}},
\oauthor{\bsnm{Constantinescu}, \binits{E.M.}},
\oauthor{\bsnm{Dalcin}, \binits{L.}},
\oauthor{\bsnm{Dener}, \binits{A.}},
\oauthor{\bsnm{Eijkhout}, \binits{V.}},
\oauthor{\bsnm{Faibussowitsch}, \binits{J.}},
\oauthor{\bsnm{Gropp}, \binits{W.D.}},
\oauthor{\bsnm{Hapla}, \binits{V.}},
\oauthor{\bsnm{Isaac}, \binits{T.}},
\oauthor{\bsnm{Jolivet}, \binits{P.}},
\oauthor{\bsnm{Karpeev}, \binits{D.}},
\oauthor{\bsnm{Kaushik}, \binits{D.}},
\oauthor{\bsnm{Knepley}, \binits{M.G.}},
\oauthor{\bsnm{Kong}, \binits{F.}},
\oauthor{\bsnm{Kruger}, \binits{S.}},
\oauthor{\bsnm{May}, \binits{D.A.}},
\oauthor{\bsnm{McInnes}, \binits{L.C.}},
\oauthor{\bsnm{Mills}, \binits{R.T.}},
\oauthor{\bsnm{Mitchell}, \binits{L.}},
\oauthor{\bsnm{Munson}, \binits{T.}},
\oauthor{\bsnm{Roman}, \binits{J.E.}},
\oauthor{\bsnm{Rupp}, \binits{K.}},
\oauthor{\bsnm{Sanan}, \binits{P.}},
\oauthor{\bsnm{Sarich}, \binits{J.}},
\oauthor{\bsnm{Smith}, \binits{B.F.}},
\oauthor{\bsnm{Zampini}, \binits{S.}},
\oauthor{\bsnm{Zhang}, \binits{H.}},
\oauthor{\bsnm{Zhang}, \binits{H.}},
\oauthor{\bsnm{Zhang}, \binits{J.}}:
{PETS}c {W}eb page.
\url{https://petsc.org/}
(2024)
\end{botherref}
\endbibitem

\bibitem[\protect\citeauthoryear{Akerson et~al.}{2022}]{Akerson_2022}
\begin{botherref}
\oauthor{\bsnm{Akerson}, \binits{A.}},
\oauthor{\bsnm{Bourdin}, \binits{B.}},
\oauthor{\bsnm{Bhattacharya}, \binits{K.}}:
Optimal design of responsive structures.
Structural and Multidisciplinary Optimization
\textbf{65}(4)
(2022)
\doiurl{10.1007/s00158-022-03200-5}
\end{botherref}
\endbibitem

\end{thebibliography}

\end{document}